\documentclass[12pt]{amsart}
\usepackage{amsfonts}
\usepackage{amssymb, amscd, amsthm}
\usepackage{latexsym}
\usepackage{multirow}
\usepackage{amsmath}
\usepackage[all]{xy}
\usepackage{mathrsfs}
\usepackage{mathtools}
\usepackage{amsbsy}
\usepackage{amsfonts}
\usepackage{mathrsfs}
\usepackage{verbatim}
\usepackage{version}
\usepackage{color}

\newtheorem{theorem}{Theorem}[section]
\newtheorem{lemma}[theorem]{Lemma}
\newtheorem{thm}[theorem]{Theorem}
\newtheorem{prop}[theorem]{Proposition}
\newtheorem{rem}[theorem]{Remark}
\newtheorem{coro}[theorem]{Corollary}

\newtheorem{con}[theorem]{Conjecture}

\newcommand{\ra}{\rightarrow}
\newcommand{\mo}{\mathcal{O}}
\newcommand{\mf}{\mathcal{F}}
\newcommand{\mg}{\mathcal{G}}

\newcommand{\me}{\mathcal{E}}

\newcommand{\mj}{\mathcal{J}}

\newcommand{\mh}{\mathcal{H}}

\newcommand{\cl}{\mathcal{L}}
\newcommand{\mc}{\mathcal{C}}

\newcommand{\cd}{\mathcal{D}}
\newcommand{\cp}{\mathcal{P}}

\newcommand{\C}{\mathscr{C}}

\newcommand{\z}{\Theta}

\newcommand{\p}{\mathbb{P}}
\newcommand{\bz}{\mathbb{Z}}
\newcommand{\bq}{\mathbb{Q}}
\newcommand{\bc}{\mathbb{C}}
\newcommand{\br}{\mathbb{R}}

\newcommand{\kf}{\mathfrak{F}}

\newcommand{\End}{\operatorname{End}}

\newcommand{\Tor}{\operatorname{Tor}}
\newcommand{\Ima}{\operatorname{Im}}

\def\<{\langle}
\def\>{\rangle}

\begin{document}
\fontsize{12pt}{14pt} \textwidth=14cm \textheight=21 cm
\numberwithin{equation}{section}
\title{Lefschetz filtration and Perverse filtration on the compactified Jacobian.}
\author{Yao Yuan}
\thanks{The author is supported by BNSF 1262005.  }

\begin{abstract}Let $C$ be a complex integral curve with plannar singularities.  Let $J$ be the compactified Jacobian of $C$.  There are two filtrations on the cohomology group $H^*(J)$.  One is obtained by the nilpotent morphism defined by cupping a certain ample divisor on $J$, which we call the Lefschetz filtration.  To obtain the other filtration, we put $C$ into a family of curves $\mathcal{C}\rightarrow B$ so that $J$ can be embedded into a family $f:\mathcal{J}\rightarrow B$, and we let $B, \mathcal{C},\mathcal{J}$ be smooth.  Then $Rf_*(\mathbb{Q}_{\mathcal{J}})$ decomposes into a direct sum of its (shifted) perverse cohomologies.  Restricting this decomposition to fibers, we get a filtration on $H^*(J)$ called the perverse filtration.  We show in this paper that these two filtrations are opposite to each other as conjectured by Maulik-Yun.

~~~


\end{abstract}

\maketitle
\tableofcontents
\section{Introduction.}
\subsection{The Lefschetz filtration.}\label{LF}
Let $M$ be a smooth complex projective variety of dimension $n$ with $\omega$ an ample class in $ H^2(M,R)$\footnote{It suffices to let $M$ be a compact K\"ahler manifold with $\omega$ a K\"ahler class}, where $R=\bq/\br/\bc$.  Then $\omega$ induces a nilpotent endomorphism $L_{\omega}$ on $H^*(M)=\bigoplus_{k=0}^{2n} H^k(M,R)$ by cup product, i.e. $L_{\omega}(a)=a\cup \omega$ for any $a\in H^*(M,R)$.  By \cite[Proposition 1.6.1]{Del} there is a unique increasing filtration 
\begin{equation}\label{Leff}W_0H^*(M)\subset W_1H^*(M)\subset\cdots\subset W_{2n-1}H^*(M)\subset W_{2n}H^*(M)=H^*(M),\end{equation}
satisfying that (1) $L_{\omega}W_k\subset W_{k-2}$ and (2) the induced endomorphism on graded pieces $L_{\omega}^k:\text{Gr}_{n+k}^WH^*(M)\xrightarrow{\cong} \text{Gr}_{n-k}^WH^*(M)$ is an isomorphism for every $0\leq k\leq n$.   By Hard Lefschetz theorem we have 
\begin{equation}\label{Lefd}W_kH^*(M)=\bigoplus_{i\geq 2n-k}H^i(M,R).\end{equation}
Therefore $\text{Gr}_{n+k}^WH^*(M)=H^{n-k}(M,R)$ and cupping with $\omega^k$ gives an isomorphism $H^{n-k}(M,R)\cong H^{n+k}(M,R)$. 
We call the filtration in (\ref{Leff}) the \emph{Lefschetz filtration}.

By the Jacobson-Morozov theorem, we can find another endomorphism $\Lambda$ of $H^*(M)$ such that $(L_{\omega},H:=[L_{\omega},\Lambda],\Lambda)$ form an $sl_2$-triple.  We then have the weight decomposition 
\[H^*(M)=\bigoplus_{\lambda=-n}^nH_{\lambda}\]  
such that $H(a)=\lambda\cdot a$ for any $a\in H_{\lambda}$.  Denote by $PD:H^{n-k}(M,R)\xrightarrow{\cong} H^{n+k}(M,R)^{\vee}$ the Poincar\'e duality.   $PD$ and the isomorphism $H^{n-k}(M,R)\cong H^{n+k}(M,R)$ given by cupping $\omega^k$ give us an isomorphism $\widetilde{PD}: H^{k}(M,R)\xrightarrow{\cong} H^{k}(M,R)^{\vee}$,   We actually have that 
\[\Lambda=-\widetilde{PD}^{-1}\circ L_{\omega}^{\vee}\circ \widetilde{PD},~\text{and}~H_{\lambda}=H^{\lambda+n}(M,R),\]
where $L_{\omega}^{\vee}:H^k(M,R)^{\vee}\ra H^{k-2}(M,R)^{\vee}$ is the dual of $L_{\omega}$.
\begin{rem}\label{Lefs} For $M$ any projective variety not necessarily smooth and $\omega$ an ample class, we still have the filtration in (\ref{Leff}) by the nilpotency of $L_{\omega}$.  However (\ref{Lefd}) may not hold in general: for instance we may not have $\dim H^{n-k}(M,R)=\dim H^{n+k}(M,R)$.
\end{rem}
Now let $M=J_X^d$ be the Jacobian of a smooth complete curve $X$ parametrizing line bundles of degree $d$ on $X$.  $J_X^d$ are all isomorphic to each other.  The following locus 
\[\Theta:=\{G\in M|H^0(G)\neq 0\}\]
forms an ample divisor on $J^{g_X-1}$ usually called the \emph{theta divisor} in the literature.  Denote by $F_X^d$ the normalized universal bundle over $X\times J_X^d$, then 
\[\mo_{J^{g_X-1}}(-\Theta)\cong \det(Rp_*F_X^{g_X-1})\] 
with $p:X\times J_X^{g_X-1}\ra J_X^{g_X-1}$ the projection.

If $C$ is integral with plannar singularities,  we have the compatified Jacobian $J_C^d$ parametrizing rank 1 torsion free sheaves on $C$ with Euler characteristic $d-g_C+1$.  $J_C^d$ are also isomorphic to each other.  We also have a universal sheaf $F^d_C$ over $C\times J_C^d$ normalized at a smooth point $x\in C$.  Let $d=2g-1$, then $p_*F_C^d=Rp_*F_C^d$ and it is a vector bundle of rank $g$ over $J_C^d$.  We define $\z:=-c_1(\det(p_*F_C^d))\in H^2(J_C^d)$ (we will see in Proposition \ref{chen} this equality holds for all $d\geq 2g-1$).  Then cupping $\z$ is still a nilpotent endomorphism, hence by \cite[Proposition 1.6.1]{Del} there is a unique increasing filtration 
\begin{equation}\label{Leffs}W_0H^*(J_C^d)\subset W_1H^*(J_C^d)\subset\cdots\subset W_{2g_C-1}H^*(J_C^d)\subset W_{2g_C}H^*(J_C^d)=H^*(J_C^d).\end{equation}

We also call the filtration (\ref{Leffs}) the Lefschetz filtration.  However, (\ref{Lefd}) may fail since $J_C^d$ might be singular.  Let $J=J_C^{0}$.
\subsection{The perverse filtrations.}\label{PF}
Let $f:X\ra Y$ be a proper morphism between two nonsingular quasi-projective varieties 
By the well-known Decomposition theorem over $\bc$ (Th\'eor\`emes 6.25, 6.2.10 in \cite{BBD}) we have 
\[Rf_{*}\bq_X\cong \bigoplus_{k\geq0} \prescript{p}{}{\mh}^{k}(Rf_{*}\bq_X)[-k],\]
where $\prescript{p}{}{\mh}^{k}(Rf_{*}\bq_X)$ are all semisimple perverse sheaves on $Y$.  Define
\begin{equation}\label{tauper}\prescript{p}{}{\tau}_{\leq i}(R\pi_{*}\bq_{X}):=\bigoplus_{j\leq i} \prescript{p}{}{\mh}^{j}(Rf_{*}\bq_X)[-j].\end{equation}
Restrict (\ref{tauper}) to each fiber $X_s$, and we get the \emph{perverse filtration} (with respect to $f$) $P_{\bullet}$ on the cohomology group $H^*(X_s,\bq)=\bigoplus_{k\geq 0}H^k(X_s,\bq)$ such that
\begin{equation}\label{pfdef}\bigoplus_{k\geq 0}P_{\leq i}H^k(X_s,\bq)[-k]=\prescript{p}{}{\tau}_{\leq i}(R\pi_{*}\bq_{X})|_s.\end{equation}
Notice that in general the perverse filtration (\ref{pfdef}) depends on the map $f:X\ra Y$ and might not be intrinsic on $X_s$.

Let $C$ be an integral curve with plannar singularities.  We can put $C$ into a family.  Let $\mc\ra B$ be a family of integral curves with plannar singularities, such that the fiber $\mc_0$ over $0\in B$ is isomorphic to $C$.  There is also a relative compactified Jacobian $\mj\ra B$ fiber-wise parametrizing rank 1 torsion-free sheaves of Euler characteristic $1-g_C$.  We ask $\mc,B,\mj$ are all smooth.  In \cite{MY} Maulik-Yun construct a canonical perverse filtration on $H^*(\mj_0)$ with $\mj_0\cong J$, which does not depend on the choice of $\mc\ra B$ as long as some conditions are satisfied for this family (see \cite[Theorem 1.1]{MY}).  They also pose the following conjecture.

\begin{con}[Conjecture 2.17 in \cite{MY}]\label{conj}Let $C$ be an integral curve with plannar singularities.   The perverse filtration $P_{\leq i}$ on $H^*(J)$ defined in (\ref{pfdef}) via a suitable family $\mj\ra B$ is opposite to the Lefschetz filtration $W_{\leq i}$ defined in (\ref{Leffs}).\end{con} 

In \cite{Ren} Rennemo gives a decomposition 
\begin{equation}\label{Ddemp}H^*(J)=\bigoplus_{k\geq 0}D_kH^*(J)\end{equation} 
and shows that $P_{\leq i}H^*(J)=\bigoplus_{k\leq i}D_kH^*(J)$ (\cite[Proposition 7.1]{Ren}).  We will give a review on Rennemo's result in \cite{Ren} in \ref{Renres}

\subsection{The main result.}
In this paper we prove Conjecture \ref{conj} for $C$ a complex curve.  Moreover we describe concretely the $sl_2$-triple associated to cupping $\z$.  Our main strategy is to use the decomposition (\ref{Ddemp}) given by Rennemo and the Fourier transform $\kf:H_*(J)\ra H_*(J)$.  However, for $C$ singular, the usual definition of Fourier transform doesn't work although we have a Poincar\'e sheaf on $J\times J$ constructed by Arinkin (\cite{Ari2}).  Luckily we manage to use bivariant theory to get a right definition of $\kf$ (see \S\ref{DFT}).

We define two operators $e,f$ on $H_*(J)$.  
\[e:H_*(J)\ra H_*(J),~~a\mapsto \z\cap a;\]
\[f:=-\kf\circ e\circ \kf^{-1}=-\kf\circ e\circ\kf_1.\]  

We denote by $e^{\vee},f^{\vee}$ the dual maps of $e,f$ respectively on $H^*(J)$.  Notice that  $e^{\vee}(\alpha)=\alpha\cup\Theta$ for any $\alpha\in H^*(J)$.

Our main result is as follows.
\begin{thm}[Corollary \ref{maincoro}]$(e^{\vee},  [e^{\vee},f^{\vee}],f^{\vee})$ forms a $sl_2$-triple on $H^*(J)$, whose eigenvalue decomposition coincides with the decomposition (\ref{Ddemp})on $H^*(J)$.  

In particular, the Lefschetz filtration induced by cupping $\z$ on $H^*(J)$ is $W_{k}(H^*(J))=\bigoplus_{j\geq 2g-k}D_jH^*(J)$ and it is opposite to the perverse filtration $P_{\leq m}H^*(J)=\bigoplus_{j\leq m} D_{j}H^*(J)$, i.e. $W_{k}(H^*(J))\cap P_{\leq m}H^*(J)={0}$ if $m+k<2g$.   
\end{thm}
\begin{rem}Why we need $C$ to be a complex curve?  There are two reasons:

(1) We only know bivariant theory over $\bc$;

(2) In the proof of Proposition \ref{chen}, we need the compactified Jacobian $J_C$ to admit a resolution of singularities.
\end{rem}

\subsection{The plan of the paper.} In \S 2.1 we review Rennemo's result and provide some relations of the Chern classes of the bundle $p_*(F_C^d)$ on $J_C^d$.  Those relations are very important to our main theorem.  In \S 2.2 we describe Rennemo's decomposition as an eigenvalue decomposition of some linear functor.  In \S 2.3 we prove some properties for the Poincar\'e sheaf which will be used to prove our main theorem.  In \S 3.1 we define the Fourier transform $\kf$ and show that it is an isomorphism on $H_*(J)$.  In \S 3.2 we prove our main theorem.  There are two appendices.  \S\ref{appB} is devoted to a brief review of bivariant theory and Riemann-Roch for singular varieties.  \S\ref{appA} gives the proof of Proposition \ref{chen}. 
 
Our paper contains many technic proofs using bivariant theory and hence the readers are suggested to read the review in the appendix before going through all those details.

\subsection{Notations, Conventions and Prelimilaries.}\label{NC}
\begin{itemize}
\item Let $C$ be a proper integral curve over $\bc$ with planar singularities.  Denote by $g$ its arithmetic genus.  Let $C^{sm}\subset C$ be the open subset consisting of smooth points of $C$.  We fix a smooth point $x\in C^{sm}$. 
\item Let $J$ be the compactified Jabobian which parametrizes torsion free rank 1 sheaves $G$ on $C$ with $\chi(G)=1-g$ (i.e. $\text{deg}(G)=0$).
\item Let $F$ be the universal sheaf over $C\times J$ normalized at $x$, i.e. $F_x\cong \mo_J$. 
\item Let $C^{[n]}$ be the Hilbert scheme of length $n$ subschemes of $C$.  Let $C^{[n,n+1]}\subset C^{[n]}\times C^{[n+1]}$ be the flag Hilbert scheme parametrizing pairs $(Z,Z')$ such that $Z\subset Z'$.  
\item We have the Abel-Jacobi map $AJ_n:C^{[n]}\ra J$ sending $Z$ to $I_Z\otimes \mo_C(nx)$ where $I_Z\subset \mo_C$ is the ideal sheaf of $Z$.  

\item By Arinkin's result (\cite[Theorem C]{Ari2}), there is a Poincar\'e sheaf $P$ on $J\times J$ which is a maximal Cohen-Macaulay sheaf.  Moreover $P$ induces the audoduality of $J$, i.e., the Fourier-Mukai functor 
\[\kf:D^b(J)\ra D^b(J);~~\mg\mapsto Rp_{1,*}(p_2^*(\mg)\otimes P)\]
is an equivalence of categories.  We may assume that $(AJ_1\times id_J)^*P\cong F$. 

\item Let $\mc\ra B$ be a family of integral curves with planar singularities such that the fiber $\mc_0$ over $0\in B$ is $C$ and $\mc_t$ is smooth for generic $t\in B$.  We may assume there is a section $s:B\hookrightarrow \mc^{sm}$ such that $s(0)=x$.  
\item For simplicity, we let $B$ have trivial tangent bundle.

\item Let $\mj,\mc^{[n]},\mc^{[n,n+1]}$ be the relative version of $J,C^{[n]},C^{[n,n+1]}$ over $B$ respectively.  We may also assume $B,\mc,\mj,\mc^{[n]},\mc^{[n,n+1]}$ are all smooth (see e.g. \cite[Corollary 15]{She}).  
Hence $C,J,C^{[n]},C^{[n,n+1]}$ are all locally complete intersections.  
\item Let $\cp$ be the relative Poincar\'e sheaf over $\mj\times_B\mj$.  
\item Let $\mf$ be the relative universal sheaf over $\mc\times_B \mj$ normalized at the section $s$.
\item We have the relative Abel-Jacobi map $\widetilde{AJ}_n:\mc^{[n]}\ra \mj$ and we may assume that $(\widetilde{AJ}_1\times_B id_{\mj})^*\cp\cong \mf$.
\item All cohomology and homology groups are of $\bq$-coefficients unless otherwise stated.
\item We denote by $f^*,f_*,\otimes$ the usual pullback, pushforward and tensor product, while we use $Lf^*,Rf_*,\otimes^L$ to denote their derived forms.
\end{itemize}
\subsection{Acknowledgements.}I thank Xuezhi Zhao and Fang Sun for their help in topology.  I thank the referees.    


\section{Some properties on $J$ and $C^{[n]}$}
\subsection{A decomposition on $H_*(J)$.}\label{Renres}
In this subsection we review some results of Rennemo in \cite{Ren}.  Let 
\[V(C):=\bigoplus_{n,k\geq 0}H_k(C^{[n]}).\]
Rennemo has defined two pairs of operators $\mu_{\pm}[pt],\mu_{\pm}[C]$ in \cite[\S 2]{Ren} on $V(C)$.  We recall the definitions of them as follows.  

We have fixed a smooth point $x\in C$, hence we have an inclusion $i:C^{[n]}\hookrightarrow C^{[n+1]}$ sending $Z$ to $Z\cup\{x\}$ and $i$ is a regular embedding by \cite[Lemma 2.1]{Ren}.  We define two maps $\mu_{\pm}[pt]:H_k(C^{[n]})\ra H_{k-1\pm1}(C^{[n\pm1]})$ by letting $\mu_{+}[pt]=i_*$ and $\mu_{-}[pt]=i^!$ where $i^!$ is the intersection pullback map. 

For the flag Hilbert scheme $C^{[n,n+1]}$, we have three natural maps $p,\sigma,q$ as in the following diagram, sending $(Z_n,Z_{n+1})$ to $Z_n,Z_{n+1}\setminus Z_{n},Z_{n+1}$ respectively.  
\begin{equation}\label{3maps}\xymatrix{&C^{[n,n+1]}\ar[ld]_{p}\ar[d]_{\sigma}\ar[rd]^{q}&\\C^{[n]}&C&C^{[n+1]}}\end{equation}
We define two maps $\mu_{\pm}[C]:H_k(C^{[n]})\ra H_{k+1\pm1}(C^{[n\pm1]})$ by letting $\mu_{+}[C]=q_*p^{!}$ and $\mu_{-}[C]=p_*q^!$, where $p^!$ and $q^!$ are refined Gysin pullbacks defined via the regular embeddings $C^{[n]}\hookrightarrow\mc^{[n]}$ and $C^{[n+1]}\hookrightarrow\mc^{[n+1]}$ respectively (see \cite[\S 6.2 and Ex. 19.1.10]{Ful}).  By \cite[Lemma 4.1]{Ren}, $p^!$ ($q^!$, resp.) is independent of the choice of $\mc^{[n]}$ ($\mc^{[n+1]}$, resp.). 

In order to pass from homology to cohomology,  we only need to dualize the vector space $V(C)$ and operators $\mu_{\pm}[pt],\mu_{\pm}[C]$.  Let 
$$V^c(C)=\bigoplus_{k,n\geq0}H^k(C^{[n]})=V(C)^{\vee},$$ 
and let $\mu_{\pm}^c[pt]=\mu_{\mp}[pt]^{\vee},~\mu_{\pm}^c[C]=\mu_{\mp}[C]^{\vee}$.  Notice that we are using the same notations as in \cite{Ren}. 

We quote some of Rennemo's results as the following two theorems.
\begin{thm}[Theorem 1.2 + Lemma 3.2 in \cite{Ren}]\label{Ren1}~~

(i) The operators $\mu_{\pm}[pt],\mu_{\pm}[C]$ satisfy the commutation relations
\[[\mu_{-}[pt],\mu_{+}[C]]=[\mu_{-}[C],\mu_{+}[pt]]=id,\]
and all other pairs of operators commute.

(ii) Let $W=\ker\mu_{-}[pt]\cap \ker\mu_{-}[C]$.  Then the natural map
\[W\otimes \bq[\mu_+[pt],\mu_+[C]]\ra V(C)\]
is an isomorphism.

(iii) The Abel-Jacobi push-forward map $AJ_*:V(C)\ra H_*(J)$ induces an isomorphism $W\cong H_*(J)$.

(iv) The Abel-Jacobi push-forward map 
\[AJ_{n,*}:\ker \mu_{-}[pt]\cap H_{*}(C^{[n]})\ra H_*(J)\]
is injective for any $n$, and is an isomorphism for $n\geq 2g$.
\end{thm}

\begin{thm}[Theorem 1.3 in \cite{Ren}]\label{Ren2}~~

(i) The operators $\mu^c_{\pm}[pt],\mu^c_{\pm}[C]$ satisfy the commutation relations
\[[\mu^c_{-}[pt],\mu^c_{+}[C]]=[\mu^c_{-}[C],\mu^c_{+}[pt]]=id,\]
and all other pairs of operators commute.

(ii) Let $W^c=V^c(C)/(\Ima\mu_{+}^c[pt]+\Ima \mu_{+}^c[C])$.  Then the natural map
\[\ker\mu^c_{-}[pt]\cap \ker\mu^c_{-}[C]\ra W^c,\]
and
\[(\ker\mu^c_{-}[pt]\cap \ker\mu^c_{-}[C])\otimes \bq[\mu_+[pt],\mu_+[C]]\ra V(C)\]
are isomorphisms.

(iii) The Abel-Jacobi pull-back map $AJ^*:H^*(J)\ra V(C)^c$ induces an isomorphism $H^*(J)\cong W^c$.
\end{thm}

\begin{rem}\label{AJinv}
We want to emphasis the explicit form of the map $(AJ_{n,*})^{-1}: H_*(J)\ra\ker \mu_{-}[pt]\cap H_{*}(C^{[n]})$ for $n\geq 2g$ in Theorem \ref{Ren2} (iv).  Since $AJ_{n}:C^{[n]}\ra J$ is a $\p^{n-g}$-bundle for $n\geq 2g-1$.  For $n\geq 2g$, we have the relative ample class $\omega_n:=c_1(\mo_{AJ_n}(1))$.  One can see that $\omega_n\cap[C^{[n]}]=[i(C^{[n-1]})]$.

Recall that we have the universal sheaf $F$ over $C\times J$ normalized at the smooth point $x\in C$.  For $n\geq 2g$, $E_n:=p_{2,*}(p_1^*\mo_C(nx)\otimes F)$ is a rank $n-g+1$ bundle over $J$, where $p_1,p_2$ are projections from $C\times J$ to $C$ and $J$ respectively.  Moreover we have an isomorphism $\p(E_n)\xrightarrow{\cong}C^{[n]}$ and the following commutative diagram
\[\xymatrix{\p(E_n)\ar[rr]^{\cong}\ar[rd]_{\pi}&&C^{[n]}\ar[ld]^{AJ_n}\\ &J&}.\]
Denote by $c_i(E_n)$ the $i$-the Chern class of $E_n$.  By the basic property of Chern classes and the fact $\dim J=g$, we have for $n\geq 2g-1$
\[\sum_{i=0}^{n-g+1}\pi^*c_i(E_n)\cup\omega_n^{n-g+1-i}=\sum_{i=0}^{g}\pi^*c_i(E_n)\cup\omega_n^{n-g+1-i}=0.\]

For any $\beta\in H_*(J)$, we have
\begin{equation}\label{preAJ} (AJ_{n,*})^{-1}(\beta)=\sum_{i=0}^{g}\pi^*c_i(E_n)\cup\omega_n^{n-g-i}\cap AJ^!(\beta)=:\gamma.
\end{equation}
We see $\gamma\in \ker\mu_{-}[pt]$, because $\mu_{+}[pt](\mu_{-}[pt](\gamma))=\omega_n\cap\gamma=0$ and $\mu_{+}[pt]$ is injective by Theorem \ref{Ren1} (ii).
 \end{rem}
The following proposition is very important for our main theorem.  The proof of Proposition \ref{chen} is a bit long and hence moved to \S\ref{appA}.
\begin{prop}\label{chen}Let $E_n,n\geq 2g-1$ be the same as in Remark \ref{AJinv}.  We have 
$$ch_i(E_n)=0\text{ for all }i\geq 2.$$
In particular $$c_i(E_n)=\frac{1}{i!}(-\Theta)^i\text{ for all }i\geq 1$$ 
with $\Theta:=-c_1(E_n)$ the theta divisor class defined in \S \ref{LF}.   
\end{prop}
Denote by $\omega_n$ the relative ample divisor class of $AJ_n:C^{[n]}\ra J$, i.e. $\omega_n=c_1(\mo_{AJ_n}(1))$.  Then $\omega_n\cap [C^{[n]}]=[i(C^{[n-1]})]$ by definition and $AJ_{n,*}(\omega_n^{n-g+1}\cap[C^{[n]}])=-c_1(E_n)\cap [J]=\z\cap[J]$.  \begin{lemma}\label{comw}Use the same notations as in (\ref{3maps}), and we have
\[q^*\omega_{n+1}=p^*\omega_{n}+\sigma^*[x]\in H^*(C^{[n,n+1]}),\]
where $[x]\in H^2(C)$ is the first Chern class of the line bundle $\mo_C(x)$.
\end{lemma}
\begin{proof}The image $i(C^{[n]})\subset C^{[n+1]}$ is a Cartier divisor and induces a line bundle which is exactly the determinant line bundle associated to $[\mo_C(x)]-[\mo_C]$, i.e. we have 
\[\mo_{C^{[n+1]}}(i(C^{[n]}))\cong \det((p_2)_*(p_1^*\mo_C(x))|_{\cd_{n+1}})\otimes \det((p_2)_*\mo_{\cd_{n+1}})^{-1},\]
where $\cd_{n+1}\subset C\times C^{[n+1]}$ is the universal subscheme and $p_1,p_2$ are projections from $C\times C^{[n+1]}$ to $C,C^{[n+1]}$ respectively.  Therefore by an analogous formula to \cite[Lemma 2.1]{EGL} we have 
\[q^*\mo_{C^{[n+1]}}(i(C^{[n]}))\cong p^*\mo_{C^{[n]}}(i(C^{[n-1]}))\otimes\sigma^*\mo_{C}(x).\]
Hence the lemma.
\end{proof}
The following lemma gives a concrete description of the theta divisor $\z$.  But we will not use it for our main theorem.

\begin{lemma}\label{nouse}$\Theta\cap[J]=[AJ_{g-1}(C^{[g-1]})]$.
\end{lemma}
\begin{proof}Choose $n-g+1$ different smooth points $x=x_1,\cdots,x_{n-g+1}\in C$, and denote by $i_{x_i}$ the regular embedding $C^{[m]}\hookrightarrow C^{[m+1]}$ sending $Z$ to $Z\cup\{x_i\}$.  Easy to see $[i_{x_i}(C^{[n-1]})]=\omega_n\cap[C^{[n]}]$ with $\omega_n$ the relative ample divisor class. Although $C^{[n]}$ is not smooth and we might not define intersection product on $H_*(C^{[n]})$, we can define the intersection product of effective Cartier divisors\footnote{more generally for cycles given by regular embedings.} as explained in \cite[Example 6.5.1]{Ful}.
Hence $\omega_n^{n-g+1}$ is represented by the image of $C^{[g-1]}$ via the regular embedding $i':C^{[g-1]}\hookrightarrow C^{[n]}$ sending $Z$ to $Z\cup\{x_1,\cdots,x_{n-g+1}\}$. 

On the other hand, $AJ_n\circ i'=T\circ AJ_{g-1}$, where $T:J\xrightarrow{\cong} J$ sends every rank 1 sheaf $G$ to $G\otimes \mo_C((n-g+1)x-\sum_{i=1}^{n-g+1}x_i)$.  Therefore
\[-c_1(E_n)\cap[J]=AJ_{n,*}(\omega_n^{n-g+1}\cap[C^{[n]}])=AJ_{n,*}([i'(C^{[g-1]})])=T_*([AJ_{g-1}(C^{[g-1]})]).\]
Finally since $T$ is a translation, $T_*([AJ_{g-1}(C^{[g-1]})])$ and $[AJ_{g-1}(C^{[g-1]})]$ are algebraically equivalent, hence $T_*([AJ_{g-1}(C^{[g-1]})])=[AJ_{g-1}(C^{[g-1]})]$ in $H^*(J)$.
\end{proof} 

The space $V^c(C)$ ($V(C)$, resp.) is equipped with a grading $D$ such that $D_n(V^c(C))=H^*(C^{[n]})$ ($D_n(V(C))=H_*(C^{[n]})$, resp.) which induces gradings on $W^c$ ($W$, resp.) and $H^*(J)$ ($H^*(J)$, resp.).  Notice that the grading $D$ on $V(C)$ induces $D$ on the subspace $W$ because $\mu_{\pm}[pt],\mu_{\pm}[C]$ are homogenous of degree $\pm1$ with respect to the grading $D$.  We call the decomposition $H^*(J)=\bigoplus_{k\geq 0}D_kH^*(J)$ induced by the grading $D$ the \emph{$D$-decomposition}.  

By \cite[Proposition 7.1]{Ren},  we have
\begin{equation}\label{D=P}\bigoplus_{k\leq m}D_kH^*(J)=:D_{\leq m}H^*(J)=P_{\leq m}H^*(J),
\end{equation}
where $P_{\leq m}$ is the perverse filtration defined in \S \ref{PF}.  Hence the $D$-decomposition $\bigoplus_m D_m H^*(J)$ provides a split of the perverse filtration.

Denote by $\mh_n:=\ker\mu_{-}[pt]\cap H_{*}(C^{[n]})$ for $n\geq 2g$.  Then by Theorem \ref{Ren1} (iv), we have an isomorphism $AJ_{n,*}:\mh_n\xrightarrow{\cong} H_*(J)$ and $(AJ_{n,*})^{-1}$ is given explicitly by (\ref{preAJ}).  The grading $D$ on $H_*(J)$ also induces a grading on $\mh_n$.  We will give a more concrete description of the induced grading on $\mh_n$ in the next subsection.  
\subsection{An eigenvalue decomposition on $\mh_n$.}
In this subsection we give a decomposition of $\mh_n$ into eigenvalue spaces of some endomorphism which coincides with the $D$-decomposition on $H_*(J)$ via Abel-Jacobi pushforward.
 
Define $\phi:=\mu_{+}[pt]\circ\mu_{-}[C]\in \End(H_*(C^{[n]}))$.  By Theorem \ref{Ren1} (i), both $\mu_{+}[pt]$ and $\mu_{-}[C]$ commute with $\mu_{-}[pt]$.  Therefore $\phi$ keeps the subspace $\mh_n$ invariant.  We also denote $\phi$ its restriction to $\mh_n$. 
\begin{prop}\label{etade}Let $n\geq 2g$.  For any $\alpha\in \mh_n$, we have 
\[AJ_{n,*}(\alpha)\in D_kH_*(J)\Leftrightarrow \phi(\alpha)=(n-k)\alpha.\]
In particular we have the eigenvalue decomposition 
\[\mh_n=\bigoplus_{i=0}^{2g}H^{\phi}_{n-i}\]
with $H_{n-i}^{\phi}$ the eigenspace of value $n-i$, and $AJ_{n,*}:\bigoplus_{k}D_kH_*(J)\ra \bigoplus_k H_{n-k}^{\phi}$ is an isomorphism of graded spaces. 
\end{prop}
We first prove the following lemma.
\begin{lemma}\label{mumu}Let $\alpha\in V(C)$, and let $l,k\in\bz_{\geq 0}$.  Then we have
\begin{enumerate}
\item If $\alpha\in \ker \mu_{-}[pt]$, then $\mu_{+}[C]^l(\alpha)\in \ker \mu_{-}[pt]^{l+1}$ and $\mu_{-}[pt]\mu_{+}[C]^{l+1}(\alpha)=(l+1)\mu_{+}[C]^{l}(\alpha)$.
\item If $\alpha\in \ker \mu_{-}[C]$, then $\mu_{+}[pt]^l(\alpha)\in \ker \mu_{-}[C]^{l+1}$ and $\mu_{-}[C]\mu_{+}[pt]^{l+1}(\alpha)=(l+1)\mu_{+}[pt]^{l}(\alpha)$.
\end{enumerate}
In particular, recall that $W= \ker \mu_{-}[pt]\cap \ker \mu_{-}[C]$, we have
\[\mu_{+}[C]^l\mu_{+}[pt]^k W\cap \ker\mu_{-}[pt]=0, \text{ if }l\neq0,\]
\[\mu_{+}[C]^l\mu_{+}[pt]^k W\cap \ker\mu_{-}[C]=0, \text{ if }k\neq0.\]
\end{lemma}
\begin{proof}(1) It is enough to prove $\mu_{-}[pt]\mu_{+}[C]^{l+1}(\alpha)=(l+1)\mu_{+}[C]^{l}(\alpha)$ and $\mu_{+}[C]^l(\alpha)\in \ker \mu_{-}[pt]^{l+1}$ will follow.  We do the induction on $l$.  

Let $l=0$, then $\mu_{-}[pt]\mu_{+}[C](\alpha)=[\mu_{-}[pt],\mu_{+}[C]](\alpha)=\alpha$ since $\alpha\in \ker\mu_{-}[pt]$ and $[\mu_{-}[pt],\mu_{+}[C]]=id$ by Theorem \ref{Ren1} (i).
For general $l$, by induction assumption we have 
\begin{eqnarray}\mu_{-}[pt]\mu_{+}[C]^{l+1}(\alpha)&=&[\mu_{-}[pt],\mu_{+}[C]](\mu_{+}[C]^{l}\alpha)+\mu_{+}[C]\mu_{-}[pt]\mu_{+}[C]^l(\alpha)\nonumber\\
&=&\mu_{+}[C]^l(\alpha)+\mu_{+}[C]l\mu_{+}[C]^{l-1}(\alpha)\nonumber\\
&=&(l+1)\mu_{+}[C]^l(\alpha).\nonumber
\end{eqnarray}

(2) is analogous to (1) and left to the readers.

Finally for any $\alpha \in W$, since $\mu_{+}[pt]^k\alpha\in\ker\mu_{-}[pt]$ we have 
\begin{eqnarray}\mu_{-}[pt](\mu_{+}[C]^l\mu_{+}[pt]^k\alpha)&=&\mu_{-}[pt]\mu_{+}[C]^l(\mu_{+}[pt]^k\alpha)\nonumber\\
&=&l\mu_{+}[C]^{l-1}\mu_{+}[pt]^k\alpha.
\end{eqnarray}
If $\mu_{+}[C]^l\mu_{+}[pt]^k\alpha\in\ker\mu_{-}[pt]$ and $l\neq 0$, then $\mu_{+}[C]^{l-1}\mu_{+}[pt]^k\alpha=0$ and hence $\mu_{+}[C]^{l}\mu_{+}[pt]^k\alpha=0$.   The second statement is analogous.  We have proved the lemma.
\end{proof}
\begin{proof}[Proof of Proposition \ref{etade}]Since we already have the $D$-decomposition on $H_*(J)$ and $AJ_{n,*}$ is an isomorphism, we only need to prove $\Rightarrow$ . 

By definition $\beta\in D_kH_*(J)\Leftrightarrow \exists~ w\in W\cap H_*(C^{[k]})$ such that $AJ_{k,*}(w)=\beta$.  By the commutativity of $\mu_{+}[pt]$ and $\mu_{-}[pt]$,  we have $\mu_{+}[pt]^{n-k}(w)\in\mh_n$.  Since $AJ_{k,*}=AJ_{n,*}\circ\mu_{+}[pt]$, we also have 
$AJ_{n,*}(\mu_{+}[pt]^{n-k}(w))=\beta$.  
Therefore, we only need to show $\phi(\mu_{+}[pt]^{n-k}(w))=(n-k)\mu_{+}[pt]^{n-k}(w)$.  By Lemma \ref{mumu} (2) we have $\mu_{-}[C]\mu_{+}[pt]^{n-k}(w)=(n-k)\mu_{+}[pt]^{n-k-1}(w)$ and hence the statement since $\phi=\mu_{+}[pt]\circ\mu_{-}[C]$.
\end{proof}

\subsection{Pull-backs of the Poincar\'e sheaves to the flag Hilbert scheme.}
We have already shown that the $D$-decompositon on $H_*(J)$ corresponds to the $\phi$-eigenvalue decomposition on $\mh_n$.  In order to obtain a better understanding of the linear operator $\phi$, we will use the Fourier transform associated to the Poincar\'e sheaf $P$ over $J\times J$.  We would like to investigate how the operator $\mu_{-}[C]$ acts on the Fourier transform.  

Let $P_n:=(AJ_n\times id_J)^*P$ the pullback of $P$ to $C^{[n]}\times J$.
Recall that we have three natural maps $p,\sigma,q$ as in the following diagram, sending $(Z_n,Z_{n+1})$ to $Z_n,Z_{n+1}\setminus Z_{n},Z_{n+1}$ respectively.  
\[\xymatrix{&C^{[n,n+1]}\ar[ld]_{p}\ar[d]_{\sigma}\ar[rd]^{q}&\\C^{[n]}&C&C^{[n+1]}}\]
\begin{prop}\label{comp}\begin{enumerate}
\item We have $(p\times id_J)^*P_n=L(p\times id_J)^*P_n$, $(\sigma\times id_J)^*F=L(\sigma\times id_J)^*F$ and $(q\times id_J)^* P_{n+1}=L(q\times id_J)^* P_{n+1}$.  Moreover those three sheaves are all maximal Cohen-Macaulay sheaves on $C^{[n,n+1]}\times J$, which are locally free in codimension 1.

~~

\item $(p\times id_J)^*P_n\otimes^L (\sigma\times id_J)^*F=(p\times id_J)^*P_n\otimes(\sigma\times id_J)^*F$, i.e. $\Tor^i((p\times id_J)^*P_n,(\sigma\times id_J)^*F)=0$ for $i>0$.

~~

\item $(q\times id_J)^*P_{n+1}\cong (p\times id_J)^*P_n\otimes(\sigma\times id_J)^*F.$
\end{enumerate}
\end{prop}
\begin{proof}Firstly, (2) follows directly from (1) and Lemma \ref{torf}.  

Since $P$ and $F$ are flat over both factors, 
we have $P_n=L(AJ_n\times id_J)^*P$, and also $(p\times id_J)^*P_n=L(p\times id_J)^*P_n$, $(\sigma\times id_J)^*F=L(\sigma\times id_J)^*F$, $(q\times id_J)^* P_{n+1}=L(q\times id_J)^* P_{n+1}$.

We prove that those four sheaves mentioned above are maximal Cohen-Macaulay.  We only deal with $P_n$ since other three are analogous.  $P_n$ is a sheaf over $C^{[n]}\times J$ which is flat over both factors.  For every $y\in C^{[n]}$, $P_n|_{{y}\times J}$ is a maximal Cohen-Macaulay sheaf on $\{y\}\times J$, therefore by \cite[Lemma 2.1]{Ari2}, $P_n$ is maximal Cohen-Macaulay.  


Denote by $J^{sm}\subset J$ the open dense subset of $J$ parametrizing line bundles on $C$.  We also have that $(C^{sm})^{[n,n+1]}$ is open dense in $C^{[n,n+1]}$ (see e.g. \cite[Lemma 6.3]{Ren}).  It is easy to see those three sheaves are locally free on $C^{[n,n+1]}\times J^{sm}\cup (C^{sm})^{[n,n+1]}\times J$ whose complement in $C^{[n,n+1]}\times J$ is of codimension $\geq 2$.  (1) is proved.

To prove (3), we only need to check the isomorphism on $C^{[n,n+1]}\times J^{sm}\cup (C^{sm})^{[n,n+1]}\times J$.  Now we restrict ourselves to the loci where $P_n,P$ and $F$ are all locally free.  By \cite[Proposition 4.5]{Ari2} or \cite[Remark 4.9]{MRV}, $P_n$ coincides with the determinant line bundle associated to the class $[F]-[\mo_{C\times J}]$, i.e. we have
\[P_n\cong \det((p_2\times id_J)_*((p_1\times id_J)^*F)|_{\cd_n\times J})\otimes \det((p_2\times id_J)_*\mo_{\cd_n\times J})^{-1},\]
where $\cd_n\subset C\times C^{[n]}$ is the universal subscheme and $p_1,p_2$ are projections from $C\times C^{[n]}$ to $C,C^{[n]}$ respectively.   Therefore by an analogous formula to \cite[Lemma 2.1]{EGL} we have 
\[(q\times id_J)^*P_{n+1}\cong (p\times id_J)^*P_n\otimes(\sigma\times id_J)^*F.\]

The proposition is proved.
\end{proof}
\begin{lemma}\label{torf}Let $X$ be a Cohen-Macaulay scheme of pure dimension.  Let $G_1,G_2$ be two maximal Cohen-Macaulay sheaves on $X$.  Assume both $G_1,G_2$ are locally free in codimension 1.  Then
\begin{enumerate}
\item $G_1\otimes G_2$ is torison-free on $X$.
\item $\Tor^i(G_1,G_2)=0$ for all $i>0$.
\end{enumerate}
 \end{lemma}
 \begin{proof}(1) Let $\jmath:U\hookrightarrow X$ be an open embedding with $X\setminus U$ of codimension 2, and we can assume both $G_1$ and $G_2$ are locally free over $U$.  
 
Since $G_1,G_2$ are maximal Cohen-Macaulay, $G_k=\jmath_*(G_k|_U)$ for $k=1,2$ (see e.g. \cite[Lemma 2.2]{Ari2}).  Because $X$ is also Cohen-Macaulay, every section in $G_1\otimes G_2(V)$ can be uniquely determined by a section in $G_1\otimes G_2(V\cap U)$, and every section of $G_1\otimes G_2(V\cap U)$ can be uniquely extended to a section in $G_1\otimes G_2(V)$.  Therefore $G_1\otimes G_2=\jmath_*(G_1\otimes G_2 |_{U})$.  Since $G_1,G_2$ are locally free on $U$, $G_1\otimes G_2$ are locally free hence torsion free on $U$.  Thus $G_1\otimes G_2=\jmath_*(G_1\otimes G_2 |_{U})$ is torsion free.

(2) Let $0\ra G_3\ra E\ra G_2\ra 0$ be an exact sequence with $E$ locally free.  Then $G_3$ is also maximal Cohen-Macaulay and locally free in codimension 1.  We have the following exact sequences
\begin{equation}\label{tor1}0\ra \Tor^1(G_1,G_2)\ra G_1\otimes G_3\ra G_1\otimes E\ra G_1\otimes G_2\ra0;\end{equation}
\begin{equation}\label{tor2}0\ra \Tor^{i+1}(G_1,G_2)\ra \Tor^{i}(G_1,G_3)\ra 0,~\text{for all }i\geq1.\end{equation}
By (1) we see $G_1\otimes G_3$ is torsion free and hence $\Tor^1(G_1,G_2)=0$ by (\ref{tor1}).  Therefore $\Tor^1(G_1,G')=0$ for any $G'$ maximal Cohen-Macaulay and locally free in codimension 1.  Hence $\Tor^2(G_1,G_2)=0$ by (\ref{tor2}) and we can obtain that $\Tor^i(G_1,G_2)=0$ for all $i>0$ by repeating the argument inductively.
\end{proof}

\begin{rem}\label{compgl}Let $X,G_1$ be as in Lemma \ref{torf}.  Assume $X$ is Gorenstein, then $G_1^{\vee}:=\mh om(G_1,\mo_X)$ is also a maximal Cohen-Macaulay sheaf locally free in codimension 1 by \cite[Lemma 7.7]{Ari2}.  Moreover, $G_1^{\vee}\otimes G_1=G_1^{\vee}\otimes^L G_1\cong \mo_X$.

Therefore after shrinking $B$ if necessary, Proposition \ref{comp} (3) can be extended to a relative version, i.e. we have on $\mc^{[n,n+1]}\times_B\mj$
\begin{equation}\label{tengl}(\widetilde{q}\times_B id_{\mj})^*\cp_{n+1}\cong (\widetilde{p}\times_B id_{\mj})^*\cp_n\otimes(\widetilde{\sigma}\times_B id_{\mj})^*\mf,
\end{equation}
where $\cp_n:=(\widetilde{AJ}_n\times_B id_{\mj})^*\cp$, and $\widetilde{p},\widetilde{q},\widetilde{\sigma}$ are as in the following diagram
\[\xymatrix{&\mc^{[n,n+1]}\ar[ld]_{\widetilde{p}}\ar[d]_{\widetilde{\sigma}}\ar[rd]^{\widetilde{q}}&\\ \mc^{[n]}&\mc&\mc^{[n+1]}}.\]
\end{rem}

\section{Proof of the main result.}
\subsection{The Fourier transform.}\label{DFT} 
In this subsection, we construct a  Fourier transform on $H_*(J)$.  

Firstly we recall that for $C$ smooth, $J$ is also smooth with $P$ a line bundle over $J\times J$.  We have the classical Fourier transform $\kf:H_*(J)\ra H_*(J)$
\begin{equation}\label{cFT}\kf(\alpha):=q_{1,*}(ch(P)\cap q_2^*(\alpha)),~\forall \alpha\in H_*(J),
\end{equation}
where $q_1,q_2:J\times J\ra J$ are two projections.  

One can also define the Fourier transform on the cohomology groups $H^*(J)$ as the dual of $\kf$ or equivalently by replacing $\cap$ by $\cup$ in (\ref{cFT}).  Fourier transform is an isomorphism on $H_*(J)$ for $C$ smooth. 

Now let $C$ be integral with plannar singularities.  Since $J$ might be singular, $P$ over $J\times J$ is only a maximal Cohen-Macaulay sheaf.  We only have the tau class $\tau(P)\in H_*(J)$ instead of the Chern polynomial.  Also the intersection product might not be well defined on $H_*(J)$.  The solution to this problem is to use \emph{bivariant theory}.  We give a brief review on this theory in \S\ref{appBT} and one can see \cite[\S I.3]{FM} and \cite[Ch 17]{Ful} for more details.  There is also a brief review in \cite[\S 4]{Ren}.  

For any morphism $f:X\ra Y$, there is a bivariant homology group denoted by $H^i(X\xrightarrow{f}Y)$.  By Remark \ref{relhom}, we identify $H^{-i}(X\ra pt)$ with the Borel-Moore homology group $H^{BM}_i(X)$ and for any $c\in H^k(X\xrightarrow{f}Y)$, we have a morphism $c:H^{BM}_i(Y)\ra H^{BM}_{i-k}(X)$ given by the product
\[H^k(X\xrightarrow{f}Y)\otimes H^{-i}(Y\ra pt)\xrightarrow{\cdot} H^{k-i}(X\ra pt).\]
For any element $a\in H^{BM}_i(Y)$, we may write $c(a)$ as an element in $H^{BM}_{i-k}(X)$ or $c\cdot a$ as an element in $H^{k-i}(X\ra pt)$.

Let $d=\dim_{\bc}\mj$.  
Since $\mj$ is smooth, by Remark \ref{relhom} there is an isomorphism $H^i(X\ra \mj)\ra H^{i-2d}(X\ra pt)$ by composed with the fundamental class $[\mj]$ in $H^{-2d}(\mj\ra pt)$. 
 



We have the relative Poincar\'e sheaf $\cp$ over $\mj\times_B\mj$.  Since $\mj\times_B\mj$ is a locally complete intersection, we have well-defined $ch^{\tau}(\cp)\in H^*(\mj\times_B\mj\ra pt)$ (see (\ref{defch})).  We use $ch^{\tau}(\cp_j)$ to denote its preimage in $H^*(\mj\times_B\mj\xrightarrow{\widetilde{q}_j}\mj)$ for $j=1,2$.  We use the same letter to denote the element in $H^i(J\xrightarrow{\imath_J} \mj)$ and its image in $H^{i-2d}(J\ra pt)$.  Therefore we have $a\cdot [\mj]=a$ for every $a\in H^i(J\ra \mj)$ and $ch^{\tau}(\cp_j)\cdot[\mj]=ch^{\tau}(\cp)$ for $j=1,2$.

We have the following Cartesian diagram
\begin{equation}\label{fourdia1}\xymatrix@C=1.5cm{J\times J\ar@{^{(}->}[r]^{\imath_{J\times J}}\ar @<0.5ex> [d]^{q_1}  \ar @<-0.5ex> [d]_{q_2}&\mj\times_B\mj\ar @<0.5ex> [d]^{\widetilde{q}_1}  \ar @<-0.5ex> [d]_{\widetilde{q}_2}\\
J\ar@{^{(}->}[r]_{\imath_{J}}& \mj}.\end{equation}

For any $a\in H^*(J\xrightarrow{\imath_J}\mj)$, the pullback $\widetilde{q}_2^*a\in H^*(J\times J\xrightarrow{\imath_{J\times J}}\mj\times_B\mj)$.
We define the \emph{Fourier transform} $\kf$ as a linear map on $H^*(J\ra pt)$
\begin{eqnarray}\label{defft}\kf:H^*(J\ra pt)&\ra& H^*(J\ra pt),\nonumber\\
a&\mapsto& q_{1,*}(\widetilde{q}_2^*(a)(ch^{\tau}(\cp)))=q_{1,*}(\widetilde{q}_2^*(a)\cdot ch^{\tau}(\cp_1))\cdot[\mj]\end{eqnarray} 

\begin{rem}\label{cp12}We also have
\[\kf(a)=q_{1,*}(\widetilde{q}_2^*(a)(ch^{\tau}(\cp)))=q_{1,*}(\widetilde{q}_2^*(a)\cdot ch^{\tau}(\cp_2)\cdot[\mj]).\]
But we can not define $q_{1,*}(\widetilde{q}_2^*(a)\cdot ch^{\tau}(\cp_2))$ because $\widetilde{q}_2^*(a)\cdot ch^{\tau}(\cp_2)\in H^*(J\times J\xrightarrow{q_2\circ\imath_{J}}\mj)$ and the map $q_2\circ\imath_{J}$ does not factor through $q_1$.
\end{rem}
The following proposition ensures that $\kf$ is well-defined.
\begin{prop}\label{indep}For any $a\in H^*(J\ra pt)$, the class $\widetilde{q}_2^*(a)(ch^{\tau}(\cp))\in H_{*}(J\times J)$ is independent of the choice of the embedding $J\hookrightarrow \mj$.
\end{prop}
\begin{proof}Our argument is analogous to \cite[\S 4.5]{Ren}.  Let $\mc_{ver}\ra B_{ver}$ be a versal deformation family of $C$ such that $\mc_{ver},~B_{ver}$ and the relative compactified Jacobian $\mj_{ver}$ are all smooth.  Denote by $\cp_{ver}$ the relative Poincar\'e sheaf over $\mj_{ver}\times_{B_{ver}}\mj_{ver}$.  
After shrinking $B$, we may have the following Cartesian diagram for $i=1,2$
\[\xymatrix@C=1.5cm{J\times J\ar@{^{(}->}[r]^{\imath_{J\times J}}\ar[d]^{q_i}  &\mj\times_{B}\mj\ar[d]^{\widetilde{q}_i}\ar[r]^{f_{\mj\times_B\mj}}&\mj_{ver}\times_{B_{ver}}\mj_{ver}\ar[d]^{\widetilde{q}_{i,ver}}\\
J\ar[d]\ar@{^{(}->}[r]_{\imath_{J}}& \mj\ar[d]\ar[r]_{f_{\mj}}&\mj_{ver}\ar[d]^{\pi_{ver}}\\ 0\ar@{^{(}->}[r]_{\imath_0}& B\ar[r]_{f_B}&B_{ver}},\]
where one sees easily that $\mj\times_B\mj\cong \mj_{ver}\times_{B_{ver}}\mj$ because $\mj\cong\mj_{ver}\times_{B_{ver}}B.$
 
Let $\dim_{\bc}\mj=d$, $\dim_{\bc}\mj_{ver}=d_{ver}$.  Then $d-d_{ver}=\dim_{\bc} B-\dim_{\bc} B_{ver}$.  Since $B$ and $B_{ver}$ are all smooth, $f_B$ is a l.c.i. morphism, hence so are $f_{\mj}$ and $f_{\mj\times_B\mj}$ because $\pi_{ver}$ and $\widetilde{q}_{i,ver}$ are flat.  Denote by $[f_{\mj}]$ ($[f_{\mj\times_B\mj}]$, resp.) the canonical orientation class (def. see Remark \ref{orien}) in $H^{2d_{ver}-2d}(\mj\xrightarrow{f_{\mj}}\mj_{ver})$ ($H^{2d_{ver}-2d}(\mj_{ver}\times_B\mj\xrightarrow{f_{\mj\times_B\mj}}\mj_{ver}\times_{B_{ver}}\mj_{ver})$, resp.).  Since $f_{\mj}$ and $f_{\mj\times_B\mj}$ are of the same codimension $d_{ver}-d$, by \cite[Proposition 17.4.1]{Ful} we have 
\begin{equation}\label{pbor}\widetilde{q}_{2,ver}^*[f_{\mj}]=[f_{\mj\times_B\mj}].\end{equation}
 
Since $\cp_{ver}$ is flat over both factors,  $\cp=(f_{\mj\times_B\mj})^*(\cp_{ver})=L(f'_{\mj})^*(\cp_{ver}))$. 
Therefore $(f_{\mj\times_B\mj})^*(ch^{\tau}(\cp_{ver}))=ch^{\tau}(\cp)$ by Theorem \ref{Ful18.3} (4), which implies that \begin{equation}\label{fjch}[f_{\mj\times_B\mj}]\cdot ch^{\tau}(\cp_{ver,i})\cdot[\mj_{ver}]=ch^{\tau}(\cp_{i})\cdot [\mj].\end{equation} 


For any $a\in H^*(J\xrightarrow{\imath_J} \mj)\cong H_*(J)$, define $a_{ver}=a\cdot [f_{\mj}]\in H^*(J\xrightarrow{f_{\mj}\circ\imath_J} \mj_{ver})\cong H_*(J)$.  Then $a_{ver}$ does not depend on the choice of $\imath:J\ra \mj$.  Combine (\ref{pbor}) and (\ref{fjch}) and we have
\begin{eqnarray}\label{acjver}\widetilde{q}_{2,ver}^*(a_{ver})\cdot ch^{\tau}(\cp_{ver,2})\cdot[\mj_{ver}]&=&(\widetilde{q}_2)^*(a)\cdot \widetilde{q}_{2,ver}^*([f_{\mj}])\cdot ch^{\tau}(\cp_{ver,1})\cdot [\mj_{ver}]\nonumber\\
&=&(\widetilde{q}_2)^*(a)\cdot ch^{\tau}(\cp_1)\cdot [\mj].
\end{eqnarray}
The proposition is proved.
\end{proof}

We have the following Cartesian diagram
 \begin{equation}\label{fourdia2}\xymatrix@C=1.5cm{J\times J\times J\ar@{^{(}->}[r]^{\imath_{J\times J\times J}}\ar @<1.0ex> [d]^{q_{12}} \ar@<0.2ex>[d]_{q_{23}} \ar @<-3.0ex> [d]_{q_{23}}&\mj\times_B\mj\times_B\mj\ar @<1.0ex> [d]^{\widetilde{q}_{12}} \ar@<0.2ex>[d]_{\widetilde{q}_{13}} \ar @<-3.0ex> [d]_{\widetilde{q}_{23}} \\J\times J\ar@{^{(}->}[r]^{\imath_{J\times J}}\ar @<0.5ex> [d]^{q_1}  \ar @<-0.5ex> [d]_{q_2}&\mj\times_B\mj\ar @<0.5ex> [d]^{\widetilde{q}_1}  \ar @<-0.5ex> [d]_{\widetilde{q}_2}\\
J\ar@{^{(}->}[r]_{\imath_{J}}& \mj}.\end{equation}
Define
\[\cp^{\vee}:=\mh om(\cp,\mo_{\mj\times_B\mj})[g]\in D^b(\mj\times_B\mj),\]
where $[-]$ is the shift functor.  By \cite[Lemma 6.2 (1)]{Ari2}, $\cp^{\vee}$ is also flat over both factors of $\mj\times_B\mj$.  We have $\widetilde{q}^*_{23}(\cp)\otimes^L \widetilde{q}^*_{12}(\cp^{\vee})=\widetilde{q}^*_{23}(\cp)\otimes \widetilde{q}^*_{12}(\cp^{\vee})$ since $\cp,\cp^{\vee}$ are flat over both factors of $\mj\times_B\mj$.
By \cite[Proposition 7.1]{Ari2}, $R\widetilde{q}_{13,*}(\widetilde{q}^*_{23}(\cp)\otimes \widetilde{q}^*_{12}(\cp^{\vee}))$ is a line bundle $\cl$ over $\mj\cong\Delta_{\mj}\subset \mj\times_B\mj$, where $\Delta_{\mj}$ is the diagonal in $\mj\times_B\mj$.
Let $\omega_{\mj}$ be the canonical bundle of $\mj$, which is also the relative canonical bundle $\omega_{\mj/B}$ of $\mj\ra B$ as $B$ has trivial tangent bundle.  We actually have $\cl\cong\omega_{\mj/B}$, but we will not use this fact.  We define 
\[\cp^{-1}:=\mh om(\cp,\mo_{\mj\times_B\mj})\otimes \widetilde{q}_1^*\cl^{\vee}[g]\in D^b(\mj\times_B\mj).\]
Then \begin{equation}\label{sedual}R\widetilde{q}_{13,*}(\widetilde{q}^*_{23}(\cp)\otimes^L \widetilde{q}^*_{12}(\cp^{-1}))=R\widetilde{q}_{13,*}(\widetilde{q}^*_{23}(\cp)\otimes \widetilde{q}^*_{12}(\cp^{-1}))\cong\mo_{\Delta_{\mj}}.\end{equation}

Define another linear map $\kf_1$ on $H_*(J)$
\begin{equation}\label{defft1}\kf_1:H_*(J)\ra H_*(J),~a\mapsto q_{1,*}(\widetilde{q}_2^*(a)(\tau(\cp^{-1}))).\end{equation} 
One can show analogously that $\kf_1$ is independent of the choice of $\imath_{J}:J\hookrightarrow \mj$.
\begin{prop}\label{inverse}$\kf_1$ is the inverse of $\kf$, i.e. $\kf_1\circ\kf=\kf\circ\kf_1=id_{H_*(J)}.$
\end{prop}
\begin{proof}We only show $\kf_1\circ\kf=id_{H_*(J)}$ and the rest is analogous and left to the readers.  The general idea of the proof is similar to the case when $C$ is smooth, but we have to handle everything in bivariant theory.  We use the same notations as in (\ref{fourdia2}).  Notice that all maps in (\ref{fourdia2}) are l.c.i..

We have the following commutative diagram for $i=1,2$
\begin{equation}\label{inpf1}\xymatrix{\mj\times_B\mj\ar@{^{(}->}[r]^{\imath_{B}}\ar[d]_{\widetilde{q}_i}&\mj\times\mj\ar[ld]^{\bar{q}_i}\\ \mj&}
\end{equation}
Let $E_{\bullet}$ be a locally free resolution of $(\imath_B)_*\cp$.  By Theorem \ref{Ful18.3} (3) and the fact that $B$ has trivial tangent bundle, there exists $ch_{\mj\times_B\mj}^{\mj\times\mj}(E_{\bullet})\in H^*(\mj\times_B\mj\xrightarrow{\imath_B}\mj\times\mj)$ such that
\begin{equation}\label{chejj}ch^{\tau}(\cp)=ch_{\mj\times_B\mj}^{\mj\times\mj}(E_{\bullet})\cdot[\mj\times\mj].
\end{equation}
Hence 
\begin{equation}\label{chejji}ch^{\tau}(\cp_i)=ch_{\mj\times_B\mj}^{\mj\times\mj}(E_{\bullet})\cdot[\bar{q}_i],
\end{equation}
where $[\bar{q}_i]\in H^*(\mj\times\mj\xrightarrow{\bar{q}_i}\mj)$ is the canonical orientation class.

We have 
\begin{eqnarray}\label{chell}\kf_1\circ\kf(a)&=&q_{1,*}\big\{\widetilde{q}_2^*(q_{1,*}(\widetilde{q}_2^*(a)\cdot ch^{\tau}(\cp_1)))(\tau(\cp^{-1}))\big\}\nonumber\\
&=&q_{1,*}\big\{q_{12,*}(\widetilde{q}_{2}^*(\widetilde{q}_2^*(a)\cdot ch^{\tau}(\cp_1)))(\tau(\cp^{-1}))\big\}\nonumber\\
&=&q_{1,*}\big\{q_{12,*}((\widetilde{q}_{2}^*(\widetilde{q}_2^*(a)\cdot ch^{\tau}(\cp_1)))\cdot\tau(\cp_2^{-1})\cdot [\mj])\big\}\nonumber\\
&=&q_{1,*}q_{12,*}\big\{(\widetilde{q}_{2}^*(\widetilde{q}_2^*(a)\cdot ch^{\tau}(\cp_1))\cdot\tau(\cp_1^{-1}) \cdot [\mj]\big\}\nonumber\\
&=&q_{1,*} q_{12,*}\{(\widetilde{q}_2\circ\widetilde{q}_{23})^*(a)\cdot \widetilde{q}_{2}^*(ch^{\tau}(\cp_1))\cdot\tau(\cp_1^{-1}) \cdot [\mj]\big\}
\end{eqnarray}
where the second equality is because of the commutativity of push-forward and pull-back (\S\ref{appBT} ($A_6$)), the third equality is because of the associativity of products (\S\ref{appBT} ($A_1$)) and $\tau(\cp_i^{-1})\in H^*(\mj\times_B\mj\xrightarrow{\widetilde{q}_i}\mj)$ satisfying $\tau(\cp_i^{-1})\cdot[\mj]=\tau(\cp^{-1})$, the forth equality is because of the commutativity of product and push-forward (\S\ref{appBT} ($A_4$)),  and finally the fifth equality is because of the functoriality of pull-backs (\S\ref{appBT} ($A_3$)) and the commutativity of product and pull-back (\S\ref{appBT} ($A_5$)).  Notice that $\widetilde{q}_{2}^*(ch^{\tau}(\cp_1))\in H^*(\mj\times_B\mj\times_B\mj\xrightarrow{\widetilde{q}_{12}}\mj\times_B\mj)$ because we have the following Cartesian diagram
\begin{equation}\label{inpf2}\xymatrix{\mj\times_B\mj\times_B\mj\ar[r]^{\quad\widetilde{q}_{12}}\ar[d]_{\widetilde{q}_{23}}&\mj\times_B\mj\ar[d]^{\widetilde{q}_2}\\ \mj\times_B\mj\ar[r]_{\widetilde{q}_1}&\mj}.
\end{equation}

We write another Cartesian diagram
\begin{equation}\label{inpf3}\xymatrix{\mj\times_B\mj\times_B\mj\ar[d]_{\widetilde{q}_{23}}\ar@{^{(}->}[r]^{\bar{\imath}_{B}}&\mj\times_B\mj\times \mj\ar[r]^{\quad\bar{q}_{12}}\ar[d]_{\bar{q}_{23}}&\mj\times_B\mj\ar[d]^{\widetilde{q}_2}\\ \mj\times_B\mj\ar@{^{(}->}[r]^{\imath_{B}}\ar[d]_{\widetilde{q}_2}&\mj\times\mj\ar[ld]^{\bar{q}_1}\ar[r]_{\bar{q}_1}&\mj\\ \mj& &},
\end{equation}
By (\ref{chejji}) and \S\ref{appBT} ($A_5$) we have 
\begin{eqnarray}\label{chejjp}\widetilde{q}_2^*(ch^{\tau}(\cp_1))=\widetilde{q}_2^*(ch_{\mj\times_B\mj}^{\mj\times\mj}(E_{\bullet})\cdot[\bar{q}_1])&=&\bar{q}_{23}^*(ch_{\mj\times_B\mj}^{\mj\times\mj}(E_{\bullet}))\cdot \widetilde{q}_{2}^*[\bar{q}_1]\nonumber\\
&=&\bar{q}_{23}^*(ch_{\mj\times_B\mj}^{\mj\times\mj}(E_{\bullet}))\cdot [\bar{q}_{12}].
\end{eqnarray}

Since $\widetilde{q}_{23}$ and $\bar{q}_{23}$ are flat,  
$\bar{q}_{23}^*E_{\bullet}$ gives a locally free resolution of $\bar{q}_{23}^*\imath_{B,*}(\cp)\cong \bar{\imath}_{B,*}\widetilde{q}_{23}^*(\cp)$, hence by \cite[Theorem 18.1]{Ful} we have
\begin{equation}\label{chejje}\bar{q}_{23}^*(ch_{\mj\times_B\mj}^{\mj\times\mj}(E_{\bullet}))=ch_{\mj\times_B\mj\times_B\mj}^{\mj\times_B\mj\times\mj}(\bar{q}_{23}^*E_{\bullet}).\end{equation}

Since $\bar{q}_{12}$ is flat and l.c.i., by Theorem \ref{Ful18.3} (4) and Remark \ref{orien} we have
\begin{equation}\label{chejjpp}[\bar{q}_{12}]\cdot \tau(\cp^{-1}_1)\cdot [\mj]=[\bar{q}_{12}](\tau(\cp^{-1}))=\bar{q}_{12}^*(\tau(\cp^{-1}))=td(-T_{\mj_3})\cap \tau(\bar{q}_{12}^*\cp^{-1})
\end{equation} 
where we denote $T_{\mj_3}$ the pull-back of the tangent bundle of the 3rd factor $\mj_3$ to $\mj\times_B\mj\times\mj$.

Combine (\ref{chejjp}), (\ref{chejje}) and (\ref{chejjpp}) and we have
\begin{eqnarray}\label{chepp}\widetilde{q}_2^*(ch^{\tau}(\cp_1))\cdot\tau(\cp^{-1}_1)\cdot [\mj]&=&ch_{\mj\times_B\mj\times_B\mj}^{\mj\times_B\mj\times\mj}(\bar{q}_{23}^*E_{\bullet})(td(-T_{\mj_3})\cap \tau(\bar{q}_{12}^*\cp^{-1}))\nonumber\\
&=&td(-T_{\mj_3})\cap ch_{\mj\times_B\mj\times_B\mj}^{\mj\times_B\mj\times\mj}(\bar{q}_{23}^*E_{\bullet})(\tau(\bar{q}_{12}^*\cp^{-1})),
\end{eqnarray}
where the second equality is because of \cite[Proposition 17.3.2]{Ful} and we also denote $T_{\mj_3}$ the pull-back of the tangent bundle of $\mj_3$ to $\mj\times_B\mj\times_B\mj$.

Since $\cp^{-1}$ is flat over both factors, $\mh^i(\bar{q}_{23}^*E_{\bullet}\otimes\bar{q}_{12}^*\cp^{-1})=0$ for $i>0$ and $\mh^0(\bar{q}_{23}^*E_{\bullet}\otimes\bar{q}_{12}^*\cp^{-1})\cong \widetilde{q}_{23}^*(\cp)\otimes\widetilde{q}_{12}^*(\cp^{-1})$.  By Theorem \ref{Ful18.3} (3) we have
\begin{equation}\label{chett}\tau(\widetilde{q}_{23}^*(\cp)\otimes\widetilde{q}_{12}^*(\cp^{-1}))=ch_{\mj\times_B\mj\times_B\mj}^{\mj_1\times_B\mj_3\times\mj_3}(\bar{q}_{23}^*E_{\bullet})(\tau(\bar{q}_{12}^*\cp^{-1}))\end{equation}

Therefore by (\ref{chell}), (\ref{chepp}) and (\ref{chett}) we have
\begin{eqnarray}\kf_1\circ\kf(a)&=&q_{1,*}q_{12,*}\big\{(\widetilde{q}_2\circ\widetilde{q}_{23})^*(a)\cdot \widetilde{q}_2^*(ch^{\tau}(\cp_1)) \cdot\tau(\cp_1^{-1})\cdot[\mj]  \big\}\nonumber\\
&=&q_{1,*}q_{12,*}\big\{(\widetilde{q}_2\circ\widetilde{q}_{23})^*(a)\cdot( td(-T_{\mj_3})\cap\tau(\widetilde{q}_{23}^*(\cp)\otimes\widetilde{q}_{13}^*(\cp^{-1})))\big\}\nonumber\\
&=&(q_{1}\circ q_{12})_*\big\{(\widetilde{q}_2\circ\widetilde{q}_{23})^*(a)\cdot( td(-T_{\mj_3})\cap\tau(\widetilde{q}_{23}^*(\cp)\otimes\widetilde{q}_{12}^*(\cp^{-1})))\big\}\nonumber\\
&=&(q_{1}\circ q_{13})_*\big\{(\widetilde{q}_2\circ\widetilde{q}_{13})^*(a)\cdot( td(-T_{\mj_3})\cap\tau(\widetilde{q}_{23}^*(\cp)\otimes\widetilde{q}_{12}^*(\cp^{-1})))\big\}\nonumber\\
&=&q_{1,*}q_{13,*}\big\{\widetilde{q}_{13}^*\widetilde{q}_2^*(a)\cdot( td(-T_{\mj_3})\cap\tau(\widetilde{q}_{23}^*(\cp)\otimes\widetilde{q}_{12}^*(\cp^{-1})))\big\}\nonumber\\
&=&q_{1,*}\big\{\widetilde{q}_2^*(a)\cdot\widetilde{q}_{13,*}( td(-T_{\mj_3})\cap\tau(\widetilde{q}_{23}^*(\cp)\otimes\widetilde{q}_{12}^*(\cp^{-1})))\big\}\nonumber\\
&=&q_{1,*}(\widetilde{q}_2^*(a)([\Delta_{\mj}]))=a.
\end{eqnarray}
where the third equality is because of the functoriality of push-forwards (\S\ref{appBT} ($A_2$)),  the sixth equality is because of the projection formula (\S\ref{appBT} ($A_7$)), and the seventh equality is because of (\ref{sedual}), Theorem \ref{Ful18.3} (1) and the fact $\tau(\mo_{\Delta_{\mj}})=td(T_{\mj_3})\cap\Delta_{\mj}$, and the last equality holds by Lemma \ref{ada}.

The proposition is proved. 
\end{proof}

\begin{lemma}\label{ada}For any $a\in H^*(J\ra pt)$, we have
\[q_{1,*}(\widetilde{q}_2^*(a)([\Delta_{\mj}]))=a.\]
\end{lemma}
\begin{proof}It suffices to show that 
\[q_{1,*}(\widetilde{q}_2^*(a)\cdot [\Delta_{\mj,1}]\cdot [\mj])=a,\]
where for $i=1,2$, $[\Delta_{\mj,i}]\in H^*(\mj\times_B\mj\xrightarrow{\widetilde{q}_i}\mj)$ is the preimage of $[\Delta_{\mj}]$ under the isomorphism $H^*(\mj\times_B\mj\xrightarrow{\widetilde{q}_i}\mj)\ra H^*(\mj\times_B\mj\ra pt)$. 

We have the following Cartesian diagram
\begin{equation}\label{fourdia3}\xymatrix@C=1.5cm{J\ar@{^{(}->}[r]^{\imath_{J}}\ar@{^{(}->}[d]_{\Delta}&\mj\ar@{^{(}->}[d]^{\widetilde{\Delta}}\\J\times J\ar@{^{(}->}[r]^{\imath_{J\times J}}\ar @<0.5ex> [d]^{q_1}  \ar @<-0.5ex> [d]_{q_2}&\mj\times_B\mj\ar @<0.5ex> [d]^{\widetilde{q}_1}  \ar @<-0.5ex> [d]_{\widetilde{q}_2}\\
J\ar@{^{(}->}[r]_{\imath_{J}}& \mj}.
\end{equation}

We have $[\Delta_{\mj,i}]=\widetilde{\Delta}_*[\widetilde{q}_i\circ\widetilde{\Delta}]\in H^*(\mj\times_B\mj\xrightarrow{\widetilde{q}_i}\mj)$ and by \S\ref{appBT} ($A_7$) and ($A_3$) we have
\[\widetilde{q}^*_2(a)\cdot \widetilde{\Delta}_*[\widetilde{q}_i\circ\widetilde{\Delta}]=\Delta_*(\widetilde{\Delta}^*\widetilde{q}^*_2(a)\cdot [\widetilde{q}_i\circ\widetilde{\Delta}])=\Delta_*((\widetilde{q}_2\circ\widetilde{\Delta})^*(a)\cdot [\widetilde{q}_i\circ\widetilde{\Delta}]).\]
Since $\widetilde{q}_i\circ \widetilde{\Delta}=id_{\mj}$, we have 
\[\widetilde{q}_2^*(a)\cdot [\Delta_{\mj,i}]\cdot [\mj]=\Delta_*((id_{\mj})^*(a)\cdot [id_{\mj}]\cdot [\mj])=\Delta_*(a\cdot [\mj]).\]
Finally by \S\ref{appBT} ($A_2$) and ($A_4$), we have
\begin{eqnarray}q_{1,*}(\widetilde{q}_2^*(a)([\Delta_{\mj}]))&=&q_{1,*}(\widetilde{q}_2^*(a)\cdot [\Delta_{\mj,i}]\cdot [\mj])\nonumber\\&=&q_{1,*}\Delta_*(a\cdot [\mj])\nonumber\\&=&q_{1,*}\Delta_*(a)\cdot [\mj]\nonumber\\&=&(q_{1}\circ\Delta)_*(a)\cdot [\mj]\nonumber\\&=&(id_J)_*(a)\cdot[\mj]=a\cdot[\mj]=a.\nonumber\end{eqnarray}
\end{proof}

\subsection{The main result.}
We define two operators $e,f$ on $H_*(J)$.  
\[e:H_*(J)\ra H_*(J),~~a\mapsto \z\cap a;\]
\[f:=-\kf\circ e\circ \kf^{-1}=-\kf\circ e\circ\kf_1.\]  
\begin{thm}\label{main}For any $\beta\in H_*(J)$, the following statements are equivalent.
\begin{enumerate}
\item $\beta\in D_kH_*(J)$;
\item $\gamma=(AJ_{n,*})^{-1}(\beta)\in H^{\phi}_{n-k}\subset \mh_n$ for any $n\geq 2g$;
\item $[f,e](\beta):=f\circ e(\beta)-e\circ f(\beta)=(k-g)\beta$;
\item $f\circ e^i(\beta)-e^i\circ f(\beta)=i(k-g-i+1)e^{i-1}(\beta)$ for all $i=1,\cdots,g$.
\end{enumerate}
\end{thm}

By Theorem \ref{main} we see that if $\beta\in D_kH_*(J)$, then $e(\beta)\in D_{k-2}H_*(J)$.  The following corollary is straightforward.

\begin{coro}$(f,[f,e],e)$ forms a $sl_2$-triple on $H_*(J)$, whose eigenvalue decomposition coincides with the $D$-decomposition on $H_*(J)$.     
\end{coro}
We denote by $e^{\vee},f^{\vee}$ the dual maps of $e,f$ respectively on $H^*(J)$.  Notice that $e^{\vee}(\alpha)=\alpha\cup\z$ for any $\alpha\in H^*(J)$, $e^{\vee}(D_kH^*(J))\subset D_{k+2}H^*(J)$ and $([f,e])^{\vee}=[e^{\vee},f^{\vee}]$.
\begin{coro}\label{maincoro}$(e^{\vee}, [e^{\vee},f^{\vee}], f^{\vee})$ forms a $sl_2$-triple on $H^*(J)$, whose eigenvalue decomposition coincides with the $D$-decomposition on $H^*(J)$.  

In particular, the Lefschetz filtration induced by cupping $\z$ on $H^*(J)$ is $W_{k}(H^*(J))=\bigoplus_{j\geq 2g-k}D_jH^*(J)$ and it is opposite to the perverse filtration $P_{\leq m}H^*(J)=\bigoplus_{j\leq m} D_{j}H^*(J)$ (see (\ref{D=P})), i.e. $W_{k}(H^*(J))\cap P_{\leq m}H^*(J)=\{0\}$ if $m+k<2g$.   
\end{coro}
\begin{rem}Corollary \ref{maincoro} proves Conjecture 2.17 in \cite{MY} for curves over $\bc$.
\end{rem}
\begin{proof}[Proof of Theorem \ref{main}](1)$\Leftrightarrow$(2) is by Proposition \ref{etade}.

(4)$\Rightarrow$(3) is obvious since (3) is the special case of (4) for $i=1$.   

To prove (2) $\Leftrightarrow$ (3), it suffices to show (2) $\Rightarrow$ (3), since we have the total decomposition $\mh_n=\bigoplus_{i=0}^{2g}H_{n-i}^{\phi}$ and $AJ_{n,*}:\mh_n\ra H_*(J)$ is an isomorphism.  Analogously, to prove (2) $\Leftrightarrow$ (4), it suffices to show (2) $\Rightarrow$ (4).  Therefore we only need to show (2) $\Rightarrow$ (4).

Recall that $\phi=\mu_{+}[pt]\circ\mu_{-}[C]$.  By (\ref{preAJ}) and Proposition \ref{chen} we have 
\begin{eqnarray}\gamma=(AJ_{n+1,*})^{-1}(\beta)&=&\sum_{j=0}^{g}(-1)^j\frac1{j!}\omega_{n+1}^{n+1-g-j}\cap AJ_{n+1}^!(\z^j\cap \beta)\nonumber\\
&=&\sum_{j=0}^{g}(-1)^j\frac1{j!}\omega_{n+1}^{n+1-g-j}\cap AJ_{n+1}^!(e^j( \beta)).\end{eqnarray}

By Proposition \ref{mu-c}, we have\footnotesize
\begin{eqnarray}&&\phi(\gamma)=\mu_{+}[pt]\mu_{-}[C](\gamma)\nonumber\\
&=&\mu_{+}[pt]\left\{\sum_{j=0}^{g}(-1)^j\frac1{j!}\big\{\omega_{n}^{n+1-g-j}\cap AJ_{n}^!(f(e^j(\beta)))+(n+1-g-j)\omega_n^{n-g-j}\cap AJ_n^!(e^j(\beta))\big\}\right\}\nonumber
\end{eqnarray} \normalsize
Since $\mu_{+}[pt]=i_*$, $AJ_n^!=i^!AJ_{n+1}^!$ and $i_*(\omega_n^k)=\omega_{n+1}^{k+1}$, we have  
\footnotesize
\begin{eqnarray}&&\phi(\gamma)=\mu_{+}[pt]\mu_{-}[C](\gamma)\nonumber\\
&=&\sum_{j=0}^{g}(-1)^j\frac1{j!}\left(\omega_{n+1}^{n+2-g-j}\cap AJ_{n+1}^!(f(e^j(\beta)))+(n+1-g-j)\omega_{n+1}^{n+1-g-j}\cap AJ_{n+1}^!(e^j(\beta))\right)\nonumber
\end{eqnarray} \normalsize
Since $\omega_{n+1}^{n+2-g}=\sum_{j=1}^g(-1)^{j-1}\frac1{j!}AJ_{n+1}^*(\z^j)\cup \omega_{n+1}^{n-g+2-j}$, we have 
\footnotesize
\begin{eqnarray}&&\phi(\gamma)=\mu_{+}[pt]\mu_{-}[C](\gamma)\nonumber\\
&=&\sum_{j=0}^{g-1}\omega_{n+1}^{n+1-g-j}\cap AJ_{n+1}^!\left( \frac{(-1)^j}{j!}(n+1-g-j)e^j(\beta)+\frac{(-1)^{j+1}}{(j+1)!}f(e^{j+1}(\beta))+\frac{(-1)^j}{(j+1)!}e^{j+1}(f(\beta))\right)\nonumber\\
&&+\frac{(-1)^g}{g!}(n+1-2g)\omega_{n+1}^{n+1-2g}\cap AJ_{n+1}^!(e^g(\beta))\nonumber
\end{eqnarray} \normalsize
Now assume $\gamma=(AJ_{n+1,*})^{-1}(\beta)\in H_{n+1-k}^{\phi}$, then $\phi(\gamma)=(n+1-k)\gamma$ by definition.  Then 
\footnotesize
\begin{eqnarray}\label{fineq}&&0=\phi(\gamma)-(n+1-k)\gamma\nonumber\\
&=&\sum_{j=0}^{g-1}\omega_{n+1}^{n+1-g-j}\cap AJ_{n+1}^!\left( \frac{(-1)^j}{j!}(k-g-j)e^j(\beta)+\frac{(-1)^{j+1}}{(j+1)!}f(e^{j+1}(\beta))+\frac{(-1)^j}{(j+1)!}e^{j+1}(f(\beta))\right)\nonumber\\
&&+\frac{(-1)^g}{g!}(k-2g)\omega_{n+1}^{n+1-2g}\cap AJ_{n+1}^!(e^g(\beta))\nonumber\\
\end{eqnarray} \normalsize
By (\ref{fineq}) we see that if $e^g(\beta)\neq 0$, then $\beta=[J]$ and $k=2g$.  Also for $j=0,1,\cdots, g-1$, we have
\[ (k-g-j)(j+1)e^j(\beta)=f(e^{j+1}(\beta))-e^{j+1}(f(\beta)).\]

Therefore (2) $\Rightarrow$ (4) and the theorem is proved.
\end{proof}

\begin{prop}\label{mu-c}For any $\beta\in H_*(J)$, we have
\[\mu_{-}[C](\omega_{n+1}^k\cap AJ_{n+1}^{!}(\beta))=\omega_{n}^k\cap AJ_n^!(\kf(-\z\cap\kf^{-1}(\beta)))+k\omega_n^{k-1}\cap AJ_{n}^!(\beta).\]
\end{prop}
\begin{proof}We have the following diagrams
\[\xymatrix{&C^{[n,n+1]}\ar[ld]_p\ar[rd]^q\ar[d]^{\sigma}&&&\mc^{[n,n+1]}\ar[ld]_{\widetilde{p}}\ar[rd]^{\widetilde{q}}\ar[d]^{\widetilde{\sigma}}&\\ C^{[n]}&C&C^{[n+1]},&\mc^{[n]}&\mc&\mc^{[n+1]}.}\]
Recall that $\mu_{-}[C]=p_*\circ q^!$ and $q^!$ is defined as the refined Gysin pull-back via the regular embedding $\imath_{C^{[n+1]}}:C^{[n+1]}\hookrightarrow \mc^{[n+1]}$.  To be precise, we identify $H_*(C^{[n+1]})$ with $H^*(\mc^{[n+1]},\mc^{[n+1]}\setminus C^{[n+1]})$ and there exists a refined intersection product
\[-\times-:H_k(C^{[n+1]})\otimes H_l^{BM}(\mc^{n,n+1})\ra H_{k+l-2\dim \mc^{[n]}}(C^{[n,n+1]}).\]
For any $b\in H_*(C^{[n+1]})$, $q^!(b):=b\times [\mc^{[n,n+1]}]$.  This definition can be rewritten in bivariant theory since $H^*(\mc^{[n+1]},\mc^{[n+1]}\setminus C^{[n+1]})\cong H^*(C^{[n+1]}\xrightarrow{\imath_{C^{[n+1]}}}\mc^{[n+1]})$.  We have $q^!(b)=\widetilde{q}^*(b)\cdot[\mc^{[n,n+1]}]$ viewing $b$ as an element in $H^*(C^{[n+1]}\xrightarrow{\imath_{C^{[n+1]}}}\mc^{[n+1]})$.

Since $\omega_{n+1}$ is the first Chern class of a line bundle, by \cite[Proposition 17.3.2]{Ful}, we have
\[q^{!}(\omega_{n+1}^k\cap AJ_{n+1}^!(\beta))=q^*(\omega_{n+1})^k\cap q^!AJ_{n+1}^{!}(\beta)\]
By Lemma \ref{comw}, $q^*\omega_{n+1}=p^*\omega_n+\sigma^*[x]\in H^*(C^{[n,n+1]})$.  Hence
\begin{equation}\label{muc1}\mu_{-}[C](\omega_{n+1}^k\cap AJ_{n+1}^{!}(\beta))=\omega^k_n\cap\mu_{-}[C](\beta)+k\omega_n^{k-1}\cap p_*(\sigma^*[x]\cap q^!AJ_{n+1}^{!}(\beta)).
\end{equation}  
Since $x$ is a smooth point in $C$, $\sigma^{-1}(x)\cong C^{[n]}$, $p|_{\sigma^{-1}(x)}:\sigma^{-1}(x)\ra C^{[n]}$ is an isomorphism and $q|_{\sigma^{-1}(x)}=i\circ p$.  Define $p':=(p|_{\sigma^{-1}(x)})^{-1}$.  We can write a Cartesian diagram
\[\xymatrix@C=1.5cm{C^{[n]}\ar@{^{(}->}[r]^{\imath_{C^{[n]}}}\ar@{^{(}->}[d]_{p'}&\mc^{[n]}\ar@{^{(}->}[d]^{\widetilde{p}'}\\C^{[n,n+1]}\ar@{^{(}->}[r]^{\imath_{C^{[n,n+1]}}}\ar[d]_{q}&\mc^{[n,n+1]}\ar[d]^{\widetilde{q}}\\
C^{[n+1]}\ar@{^{(}->}[r]_{\imath_{C^{[n+1]}}}& \mc^{[n+1]}}.
\]

Therefore by $p\circ p'=id_{C^{[n]}}$, $i=q\circ p'$ and $\widetilde{i}=\widetilde{q}\circ\widetilde{p}'$, we have
\begin{eqnarray}\label{muc2}p_*(\sigma^*[x]\cap q^!AJ_{n+1}^{!}(\beta))&=&p_*(\widetilde{q}^*(AJ_{n+1}^!(\beta))\cdot\widetilde{p}'_*[\mc^{[n]}])
\nonumber\\&=&p_*p'_*((\widetilde{p}')^*\widetilde{q}^*(AJ_{n+1}^{!}(\beta))\cdot[\mc^{[n]}])
\nonumber\\ &=&\widetilde{i}^*(AJ_{n+1}^!(\beta))\cdot[\mc^{[n]}]\nonumber\\
&=&i^{*}AJ_{n+1}^{!}(\beta)=AJ_{n}^{!}(\beta),
\end{eqnarray}
where the last line is because $i$ and $\widetilde{i}$ are regular embeddings and $AJ_{n+1}\circ i=AJ_{n}$.

Now we compute $p_*(q^!(AJ_{n+1}^!(\beta)))$.  We use the same notations as in \S\ref{DFT}.  Let $q_1^{\star}:C^{[\star]}\times J\ra C^{[\star]}$ is the projection with $\star=n,n+1$ or $(n,n+1)$.  Let $\beta=\kf(a)=q_{1,*}(\widetilde{q}_2^*(a)\cdot ch^{\tau}(\cp_1))\cdot[\mj]$.  Then
\begin{eqnarray}AJ_{n+1}^!(\beta)&=&\widetilde{AJ}_{n+1}^*(q_{1,*}(\widetilde{q}_2^*(a)\cdot ch^{\tau}(\cp_1))\cdot[\mc^{[n+1]}]\nonumber\\
&=&(q_1^{n+1})_*\widetilde{AJ}_{n+1}^*(\widetilde{q}_2^*(a)\cdot ch^{\tau}(\cp_1))\cdot[\mc^{[n+1]}]\nonumber\\
&=&(q_1^{n+1})_*\big\{(\widetilde{AJ}_{n+1}\times_B id_{\mj})^*\widetilde{q}_2^*(a)\cdot \widetilde{AJ}_{n+1}^*(ch^{\tau}(\cp_1))\big\}\cdot[\mc^{[n+1]}]\nonumber\\
&=&(q_1^{n+1})_*\big\{(\widetilde{AJ}_{n+1}\times_B id_{\mj})^*\widetilde{q}_2^*(a)\cdot ch^{\tau}((\cp_{n+1})_1)\big\}\cdot[\mc^{[n+1]}]
\end{eqnarray}
where $ch^{\tau}((\cp_{n+1})_1)\in H^*(\mc^{[n+1]}\times_B\mj\xrightarrow{\widetilde{q}^{n+1}_1}\mc^{[n+1]})$ such that $ch^{\tau}((\cp_{n+1})_1)\cdot [\mc^{[n+1]}]=ch^{\tau}(\cp_{n+1})$, and the last equality is because $AJ_{n+1}$ is smooth.  Hence
\begin{eqnarray}\label{qaj}q^!(AJ_{n+1}^!(\beta))&=&\widetilde{q}^*\left((q_1^{n+1})_*\big\{(\widetilde{AJ}_{n+1}\times_B id_{\mj})^*\widetilde{q}_2^*(a)\cdot ch^{\tau}((\cp_{n+1})_1)\big\}\right)\cdot[\mc^{[n,n+1]}]\nonumber\\
&=&(q_1^{n,n+1})_*\big\{(\widetilde{AJ}_{n+1}\circ\widetilde{q}\times_B id_{\mj})^*\widetilde{q}_2^*(a)\cdot \widetilde{q}^*ch^{\tau}((\cp_{n+1})_1)\big\}\cdot[\mc^{[n,n+1]}]\nonumber\\
&=&(q_1^{n,n+1})_*\big\{(\widetilde{AJ}_{n}\circ\widetilde{p}\times_B id_{\mj})^*\widetilde{q}_2^*(a)\cdot \widetilde{q}^*ch^{\tau}((\cp_{n+1})_1)\big\}\cdot[\mc^{[n,n+1]}]
\end{eqnarray} 
where the last equality is because $\widetilde{q}_2\circ(\widetilde{AJ}_{n+1}\circ\widetilde{q}\times_B id_{\mj})=\widetilde{q}_2\circ(\widetilde{AJ}_{n}\circ\widetilde{p}\times_B id_{\mj})$.
Hence
\begin{eqnarray}\label{pqaj}&& p_*q^!(AJ_{n+1}^!(\beta))\nonumber\\
&=&p_*(q_1^{n,n+1})_*\big\{(\widetilde{AJ}_{n}\circ\widetilde{p}\times_B id_{\mj})^*\widetilde{q}_2^*(a)\cdot \widetilde{q}^*ch^{\tau}((\cp_{n+1})_1)\big\}\cdot[\mc^{[n,n+1]}]\nonumber\\
&=&(q_1^{n})_*(p\times id_{J})_*\big\{(\widetilde{AJ}_{n}\circ\widetilde{p}\times_B id_{\mj})^*\widetilde{q}_2^*(a)\cdot \widetilde{q}^*ch^{\tau}((\cp_{n+1})_1)\big\}\cdot[\mc^{[n,n+1]}]\nonumber\\
&=&(q_1^{n})_*\big\{(\widetilde{AJ}_{n}\times_B id_{\mj})^*\widetilde{q}_2^*(a)\cdot(\widetilde{p}\times_B id_{\mj})_*(\widetilde{q}^*ch^{\tau}((\cp_{n+1})_1))\big\}\cdot[\mc^{[n,n+1]}]\nonumber\\
&=&(q_1^{n})_*\big\{(\widetilde{AJ}_{n}\times_B id_{\mj})^*\widetilde{q}_2^*(a)\cdot(\widetilde{p}\times_B id_{\mj})_*(\widetilde{q}^*ch^{\tau}((\cp_{n+1})_1))\cdot[\mc^{[n,n+1]}]\big\}.\nonumber\\
\end{eqnarray} 
By Lemma \ref{ppnc} (2) we have
\[\widetilde{q}^*ch^{\tau}((\cp_{n+1})_1)=(\widetilde{p}\times_B\widetilde{\sigma})^*ch^{\tau}((\cp_n\boxtimes \mf)_1). \]
where $\cp_n\boxtimes \mf=p_{13}^*\cp_{n}\otimes p_{23}^*\mf$ is a sheaf on $\mc^{[n]}\times_B\mc\times_B\mj$ with $p_{13}:\mc^{[n]}\times_B\mc\times_B\mj\ra\mc^{[n]}\times_B\mj$ and $p_{23}:\mc^{[n]}\times_B\mc\times_B\mj\ra\mc\times_B\mj$ two projections.
Therefore
\begin{eqnarray}\label{pbjch}&&(\widetilde{p}\times_B id_{\mj})_*(\widetilde{q}^*ch^{\tau}((\cp_{n+1})_1))\cdot[\mc^{[n,n+1]}]\nonumber\\&=&p_{13,*}(\widetilde{p}\times_B \widetilde{\sigma}\times_B id_{\mj})_*(\widetilde{q}^*ch^{\tau}((\cp_{n+1})_1)\cdot[\mc^{[n,n+1]}])\nonumber\\
&=&p_{13,*}(ch^{\tau}((\cp_n\boxtimes \mf)_1)\cdot[\mc^{[n]}\times_B\mc])
\end{eqnarray}

Recall that there exists $ch_{\mj\times_B\mj}^{\mj\times\mj}(E_{\bullet})\in H^*(\mj\times_B\mj\xrightarrow{\imath_B}\mj\times\mj)$ such that
\[ch^{\tau}(\cp)=ch_{\mj\times_B\mj}^{\mj\times\mj}(E_{\bullet})\cdot[\mj\times\mj].
\]
Let $\mc^{[n]}\times_B\mc\times_B\mj\xrightarrow{\jmath_B^{n,C}}\mc^{[n]}\times\mc\times_B\mj$ be an embedding, and let $\hat{p}_{ij}$ be the projection from $\mc^{[n]}\times\mc\times_B\mj$ to the $i,j$-th factors.  Define $E_{\bullet}^n:=(\widetilde{AJ}_{n}\times id_{\mj})^*E_{\bullet}$.  Then $\hat{p}^*_{13}E_{\bullet}^n$ is a locally free resolution of $(\jmath_B^{n,C})_*(p_{13}^*\cp_n)$ and by the flatness of $\mf$ and $\cp$ over both factors we have 
\begin{eqnarray}\label{pbjch1}ch^{\tau}((\cp_n\boxtimes \mf)_1)&=&ch_{\mc^{[n]}\times_B\mc\times_B\mj}^{\mc^{[n]}\times\mc\times_B\mj}(\hat{p}_{13}^*E^n_{\bullet})\cdot \hat{p}^*_{23}(ch^{\tau}(\mf_1))\nonumber\\
&=&\hat{p}_{13}^*ch_{\mc^{[n]}\times_B\mj}^{\mc^{[n]}\times\mj}(E^n_{\bullet})\cdot \hat{p}^*_{23}(ch^{\tau}(\mf_1)).\end{eqnarray}
 
Combine (\ref{pbjch}) and (\ref{pbjch1}) and we have
\begin{eqnarray}\label{pbjch2}&&(\widetilde{p}\times_B id_{\mj})_*(\widetilde{q}^*ch^{\tau}((\cp_{n+1})_1))\cdot[\mc^{[n,n+1]}]\nonumber\\
&=&p_{13,*}(\hat{p}_{13}^*ch_{\mc^{[n]}\times_B\mj}^{\mc^{[n]}\times\mj}(E^n_{\bullet})\cdot \hat{p}^*_{23}(ch^{\tau}(\mf_1))\cdot[\mc^{[n]}\times_B\mc])\nonumber\\
&=&ch_{\mc^{[n]}\times_B\mj}^{\mc^{[n]}\times\mj}(E^n_{\bullet})\cdot
\hat{p}_{13,*}(\hat{p}_{23}^*(ch^{\tau}(\mf_1)\cdot [\mc]))\nonumber\\
&=&ch_{\mc^{[n]}\times_B\mj}^{\mc^{[n]}\times\mj}(E^n_{\bullet})\cdot
\bar{r}_{2}^*(\widetilde{s}_{2,*}(ch^{\tau}(\mf_1)\cdot [\mc]))\end{eqnarray}
where $\bar{r}_2:\mc^{[n]}\times\mj\ra\mj$ and $\widetilde{s}_2:\mc\times_B\mj\ra\mj$ are projections.

By the relative version of (\ref{cceng}) we see that $\widetilde{s}_{2,*}(ch^{\tau}(\mf_1)\cdot [\mc])$ lifts to cohomology:  we have
\begin{equation}\label{sss22}\widetilde{s}_{2,*}(ch^{\tau}(\mf_1)\cdot [\mc])=ch_{>0}(\me_m)\cap[\mj],\end{equation}
where $\me_m$ is a vector bundle over $\mj$ whose restriction to $J$ is $E_m$ as defined in Remark \ref{AJinv}.   Combine (\ref{pqaj}), (\ref{pbjch2}) and (\ref{sss22}), and we have
\begin{eqnarray}\label{huhu}&&p_*q^!(AJ_{n+1}^!(\beta))\nonumber\\
&=&(q_1^{n})_*\big\{(\widetilde{AJ}_{n}\times_B id_{\mj})^*\widetilde{q}_2^*(a)\cdot(\widetilde{p}\times_B id_{\mj})_*(\widetilde{q}^*ch^{\tau}((\cp_{n+1})_1))\cdot[\mc^{[n,n+1]}]\big\}\nonumber\\
&=&(q_1^n)_*\big\{(\widetilde{AJ}_n\times_B id_{\mj})^*\widetilde{q}_2^*(a)\cdot ch_{\mc^{[n]}\times_B\mj}^{\mc^{[n]}\times\mj}(E^n_{\bullet})\cdot
\bar{r}_{2}^*(\widetilde{s}_{2,*}(ch^{\tau}(\mf_1)\cdot [\mc]))\}
\nonumber\\
&=&(q_1^n)_*\big\{(\widetilde{AJ}_n\times_B id_{\mj})^*\widetilde{q}_2^*(a)\cdot \widetilde{r}_2^*(ch_{>0}(\me_m))\cap ch_{\mc^{[n]}\times_B\mj}^{\mc^{[n]}\times\mj}(E^n_{\bullet})\cdot [\mc^{[n]}\times\mj]))\}\nonumber\\
&=&(q_1^n)_*\big\{(\widetilde{AJ}_n\times_B id_{\mj})^*\widetilde{q}_2^*(a)\cdot \widetilde{r}_2^*(ch_{>0}(\me_n))\cap ch^{\tau}((\cp_n)_1)\cdot[\mc^{[n]}]\}\nonumber\\
&=&(q_1^n)_*\big\{r_2^*(ch_{>0}(E_m))\cap(\widetilde{AJ}_n\times_B id_{\mj})^*\widetilde{q}_2^*(a)\cdot ch^{\tau}((\cp_n)_1)\cdot[\mc^{[n]}]\}\nonumber\\
&=&(q_1^n)_*\big\{(\widetilde{AJ}_n\times_B id_{\mj})^*\widetilde{q}_2^*((ch_{>0}(E_m))\cap a)\cdot ch^{\tau}((\cp_n)_1)\cdot[\mc^{[n]}]\}\nonumber\\
&=&(q_1^n)_*\big\{(\widetilde{AJ}_n\times_B id_{\mj})^*\widetilde{q}_2^*((-\z)\cap a)\cdot ch^{\tau}((\cp_n)_1)\cdot[\mc^{[n]}]\}\nonumber\\
\end{eqnarray} 
where $\widetilde{r}_2:\mc^{[n]}\times_B\mj\ra\mj$, $r_2:C^{[n]}\times J\ra J$ are projections, the third, the fifth and the sixth equalities are because of \cite[Proposition 17.3.2]{Ful} and the fact that $\me_m|_J=E_m$, and the last equality is because of Proposition \ref{chen}. 

To prove $(q_1^n)_*\big\{(\widetilde{AJ}_n\times_B id_{\mj})^*\widetilde{q}_2^*(b)\cdot ch^{\tau}((\cp_n)_1)\cdot[\mc^{[n]}]\}=AJ^!_{n}(\kf(b))=i^*AJ^!_{n+1}(\kf(b))$ is standard hence left to the readers.

By (\ref{muc1}), (\ref{muc2}) and (\ref{huhu}) we have
\[\mu_{-}[C](\omega^k_{n+1}\cap AJ_{n+1}^!(\beta))=\omega_{n}^k\cap AJ_n^!(\kf(-\z\cap\kf^{-1}(\beta)))+k\omega_n^{k-1}\cap AJ_{n}^!(\beta).\]

The proposition is proved.
\end{proof}

\begin{lemma}\label{ppnc}With the same notations as in Remark \ref{compgl}, on $\mc^{[n,n+1]}\times_B\mj$ we have 
\begin{itemize}
\item[(1)] In $H_*(\mc^{[n,n+1]}\times_B\mj)$ we have
\begin{equation}\label{chcpn1}ch^{\tau}((\widetilde{q}\times_B id_{\mj})^*\cp_{n+1})=(\widetilde{q}\times_B id_{\mj})^*(ch^{\tau}(\cp_{n+1})),\end{equation} 
and in $H^*(\mc^{[n,n+1]}\times_B\mj\xrightarrow{\widetilde{q}_1}\mc^{[n,n+1]})$ we have
\begin{equation}\label{chcpn11}ch^{\tau}(((\widetilde{q}\times_B id_{\mj})^*\cp_{n+1})_1)=\widetilde{q}^*ch^{\tau}((\cp_{n+1})_1),\end{equation} 
\item[(2)] There is an element, denoted by $ch^{\tau}((\cp_n\boxtimes \mf)_1)$, in $H^*(\mc^{[n]}\times_B\mc\times_B\mj\xrightarrow{\widetilde{q}_1}\mc^{[n]}\times_B\mc)$ satisfying the following two conditions:
\begin{itemize}
\item[(i)] $\widetilde{q}^*ch^{\tau}((\cp_{n+1})_1)=(\widetilde{p}\times_B\widetilde{\sigma})^*ch^{\tau}((\cp_n\boxtimes \mf)_1).$

\item[(ii)] $ch^{\tau}((\cp_n\boxtimes \mf)_1)\cdot [\mc^{[n]}\times_B\mc]=ch^{\tau}(\cp_n\boxtimes \mf)$
\end{itemize}
\end{itemize}
\end{lemma}
\begin{proof}(\ref{chcpn1}) is because of Theorem \ref{Ful18.3} (4) and the fact that $q:\mc^{[n,n+1]}\ra \mc^{[n+1]}$ is l.c.i., and (\ref{chcpn11}) follows directly.
We only need to show (2).

By (\ref{tengl}) we have 
\[(\widetilde{q}\times_B id_{\mj})^*\cp_{n+1}\cong (p\times_B id_{\mj})^*\cp_{n}\otimes (\sigma\times_B id_{\mj})^*\mf\]
Also easy to see that 
\[(p\times_B id_{\mj})^*\cp_{n}\otimes (\sigma\times_B id_{\mj})^*\mf\cong (p\times_B\sigma\times_B id_{\mj})^*(\cp_n\boxtimes \mf).\]
But unfortunately the map $p\times_B\sigma:\mc^{[n,n+1]}\ra\mc^{[n]}\times_B\mc$ might not be l.c.i. if $C$ is not smooth,  therefore Theorem \ref{Ful18.3} (4) does not apply directly.

Recall that there exists $ch_{\mj\times_B\mj}^{\mj\times\mj}(E_{\bullet})\in H^*(\mj\times_B\mj\xrightarrow{\imath_B}\mj\times\mj)$ such that
\[ch^{\tau}(\cp)=ch_{\mj\times_B\mj}^{\mj\times\mj}(E_{\bullet})\cdot[\mj\times\mj],\]
and 
\[ch^{\tau}(\cp_1)=ch_{\mj\times_B\mj}^{\mj\times\mj}(E_{\bullet})\cdot[\bar{q}_1],\]
with $\bar{q}_1$ defined in (\ref{inpf3}).

Since $\cp$ is flat over both factors, $(\widetilde{AJ}_{n}\times id_{\mj})^*E_{\bullet}$ is a locally free resolution of $(\imath_B^n)_*\cp_n$ and we have 
\[(\widetilde{AJ}_{n}\times id_{\mj})^*ch_{\mj\times_B\mj}^{\mj\times\mj}(E_{\bullet})=ch_{\mc^{[n]}\times_B\mj}^{\mc^{[n]}\times\mj}((\widetilde{AJ}_{n}\times id_{\mj})^*E_{\bullet})\in H^*(\mc^{[n]}\times_B\mj\xrightarrow{\imath_B^n}\mc^{[n]}\times\mj),\]
and 
\footnotesize
\[(\widetilde{p}\circ\widetilde{AJ}_{n}\times id_{\mj})^*ch_{\mj\times_B\mj}^{\mj\times\mj}(E_{\bullet})=ch_{\mc^{[n,n+1]}\times_B\mj}^{\mc^{[n,n+1]}\times\mj}((\widetilde{p}\circ\widetilde{AJ}_{n}\times id_{\mj})^*E_{\bullet})\in H^*(\mc^{[n,n+1]}\times_B\mj\xrightarrow{\imath_B^{n,n+1}}\mc^{[n,n+1]}\times\mj),\]
\normalsize
and moreover in $H^*(\mc^{[n]}\times_B\mc\times_B\mj\xrightarrow{\imath_B^{n,C}}\mc^{[n]}\times_B\mc\times\mj)$ we have
\[\bar{p}_{13}^*(\widetilde{AJ}_{n}\times id_{\mj})^*ch_{\mj\times_B\mj}^{\mj\times\mj}(E_{\bullet})=ch_{\mc^{[n]}\times_B\mc\times_B\mj}^{\mc^{[n]}\times_B\mc\times\mj}(\bar{p}_{13}^*(\widetilde{AJ}_{n}\times id_{\mj})^*E_{\bullet}),\]
where $\bar{p}_{13}:\mc^{[n]}\times_B \mc\times\mj\ra\mc^{[n]}\times \mj$ is the projection.

Similarly for $\mf$ there exists $ch_{\mc\times_B\mj}^{\mc\times\mj}(N_{\bullet})\in H^*(\mc\times_B\mj\xrightarrow{\imath_B^C}\mc\times\mj)$ such that
\[ch^{\tau}(\mf)=ch_{\mc\times_B\mj}^{\mc\times\mj}(N_{\bullet})\cdot[\mc\times\mj].
\]
Since $\mf,\cp$ are flat over both factors, the double complex $\bar{p}^*_{13}(\widetilde{AJ}_{n}\times id_{\mj})^*E_{\bullet}\otimes \bar{p}_{23}^*N_{\bullet}$ is a resolution of $(\imath_B^{n,C})_*(\cp_{n}\boxtimes\mf)$ on $\mc^{[n]}\times_B \mc\times\mj$.
 On the other hand by Proposition \ref{comp} (2), the double complex $(\widetilde{p}\circ\widetilde{AJ}_{n}\times id_{\mj})^*E_{\bullet}\otimes (\widetilde{\sigma}\times id_{\mj})^*N_{\bullet}$ is a resolution of $(\imath_B^{n,n+1})_*\{(\widetilde{p}\times_B id_{\mj})^*\cp_{n}\otimes (\widetilde{\sigma}\times_B id_{\mj})^*\mf\}=(\imath_B^{n,n+1})_*((\widetilde{q}\times_B id_{\mj})^*\cp_{n+1})$ on $\mc^{[n,n+1]}\times\mj.$

Let $L_{\bullet}$ be the simple complex associated to the double complex $\bar{p}^*_{13}(\widetilde{AJ}_{n}\times id_{\mj})^*E_{\bullet}\otimes \bar{p}_{23}^*N_{\bullet}$.  Then $(\widetilde{p}\times_B\widetilde{\sigma}\times id_{\mj})^*L_{\bullet}$ is the simple complex associated to $(\widetilde{p}\circ\widetilde{AJ}_{n}\times id_{\mj})^*E_{\bullet}\otimes (\widetilde{\sigma}\times id_{\mj})^*N_{\bullet}$ and by \cite[Theorem 18.1]{Ful} we have
\[(\widetilde{p}\times_B\widetilde{\sigma}\times id_{\mj})^*ch^{\mc^{[n]}\times_B\mc\times\mj}_{\mc^{[n]}\times_B\mc\times_B\mj}(L_{\bullet})=ch^{\mc^{[n,n+1]}\times\mj}_{\mc^{[n,n+1]}\times_B\mj}((\widetilde{p}\times_B\widetilde{\sigma}\times id_{\mj})^*L_{\bullet}).\]

By abuse of notations, we denote $\bar{q}_1:\bigstar\times\mj\ra\bigstar$ the projection to the first factor for $\bigstar=\mj,\mc^{[n+1]},\mc^{[n,n+1]},\mc^{[n]}\times_B\mc$.  Since $B$ has trivial tangent bundle 
we have 
\begin{eqnarray}\widetilde{q}^*ch^{\tau}((\cp_{n+1})_1)&=&ch^{\tau}(((\widetilde{q}\times_B id_{\mj})^*\cp_{n+1})_1)\nonumber\\&=&ch^{\tau}(((\widetilde{p}\times_B id_{\mj})^*\cp_{n}\otimes (\widetilde{\sigma}\times_B id_{\mj})^*\mf)_1)\nonumber\\&=&ch^{\mc^{[n,n+1]}\times\mj}_{\mc^{[n,n+1]}\times_B\mj}((\widetilde{p}\times_B\widetilde{\sigma}\times id_{\mj})^*L_{\bullet})\cdot[\bar{q}_1]\nonumber\\ &=&(\widetilde{p}\times_B\widetilde{\sigma}\times id_{\mj})^*ch^{\mc^{[n]}\times_B\mc\times\mj}_{\mc^{[n]}\times_B\mc\times_B\mj}(L_{\bullet})\cdot [\bar{q}_1]\nonumber\\
&=&(\widetilde{p}\times_B\widetilde{\sigma})^*\big\{ch^{\mc^{[n]}\times_B\mc\times\mj}_{\mc^{[n]}\times_B\mc\times_B\mj}(L_{\bullet})\cdot [\bar{q}_1]\big\}\nonumber
\end{eqnarray}
We define $ch^{\tau}((\cp_n\boxtimes \mf)_1):=ch^{\mc^{[n]}\times_B\mc\times\mj}_{\mc^{[n]}\times\mc\times\mj}(L_{\bullet})\cdot [\bar{q}_1]$, then $ch^{\tau}((\cp_n\boxtimes \mf)_1)$ satisfies (i) and (ii).  The lemma is proved
\end{proof}
\begin{rem}Since $\mc^{[n]}\times_B\mc$ might not be smooth, the map 
\[H^*(\mc^{[n]}\times_B\mc\times_B\mj\ra \mc^{[n]}\times_B\mc)\xrightarrow{\cdot[\mc^{[n]}\times_B\mc]}H^*(\mc^{[n]}\times_B\mc\times_B\mj\ra pt)\]
might not be an isomorphism.  Hence we don't know whether the choice of $ch^{\tau}((\cp_n\boxtimes \mf)_1)$ is unique.
\end{rem}

\appendix  
\section{A review on some homology theory.}\label{appB}
\subsection{Bivariant theory.}\label{appBT} We give a brief review on bivariant theory in this section.  One can find more details in \cite[\S I.3]{FM} and \cite[Ch 17]{Ful}.

Let $f:X\ra Y$, $g:Y'\ra Y$ be morphisms of reasonable topological spaces, forming the following Cartesian diagram
\[\xymatrix{X'\ar[r]^{f'}\ar[d]_{g'}&Y'\ar[d]^g\\X\ar[r]_f&Y}.\]

A \emph{bivariant class} $c$ in $H^i(X\xrightarrow{f}Y)$ is a collection of homomorphisms
\[c_g:H_{*}^{BM}(Y')\ra H_{*-i}^{BM}(X').\]
We also write $c\cap\alpha$ instead of $c_g(\alpha)$.  

There are three basic operations on the bivariant groups as follows.
\begin{itemize}
\item[$\bullet$]\emph{Product.} For all morphisms $f:X\ra Y$, $g:Y\ra Z$, and integers $p,q$, there is a homomorphism
\[H^p(X\xrightarrow{f}Y)\otimes H^q(Y\xrightarrow{g}Z)\xrightarrow{\cdot}H^{p+q}(X\xrightarrow{g\circ f}Z),\]
where the image of $c\otimes d$ is denoted by $c\cdot d$.  Given $Z'\ra Z$ form the fiber diagram
\begin{equation}\label{bhdia}\xymatrix{X'\ar[r]^{f'}\ar[d]_{g'}&Y'\ar[d]^g\ar[r]^{g'}&Z'\ar[d]\\X\ar[r]_f&Y\ar[r]_g&Z}.    \end{equation}
Then for $\alpha\in H_*(Z')$, we have $c\cdot d(\alpha)=c(d(\alpha)).$
\item[$\bullet$]\emph{Push-forward.} If $f:X\ra Y$ is a proper morphism, $g:Y\ra Z$ and morphism, and $p$ an integer, there is a homomorphism
\[f_*:H^p(X\xrightarrow{g\circ f}Z)\ra H^p(Y\xrightarrow{g}Z).\]
In the diagram (\ref{bhdia}), we have for any $c\in H^*(X\xrightarrow{g\circ f}Z),\alpha\in H_*(Z')$ 
\[f_*(c)(\alpha)=f'_*(c(\alpha))\]
\item[$\bullet$]\emph{Pull-back.} Given $f:X\ra Y$, $g:Y_1\ra Y$ form the Cartesian diagram
\[\xymatrix{X_1\ar[r]^{f_1}\ar[d]_{g_1}&Y_1\ar[d]^g\\X\ar[r]_f&Y}.\]
For each $p$ there is a homomorphism
\[g^*:H^p(X\xrightarrow{f}Y)\ra H^p(X_1\xrightarrow{f_1}Y_1).\]
Choose any $c\in H^*(X\xrightarrow{f}Y),Y'\ra Y_1$ and $\alpha\in H_*(Y')$.  Let $X'=X\times_Y Y'=X_1\times_{Y_1}Y'$, we have
\[g^*(c)(\alpha)=c(\alpha).\]
\end{itemize}
Those operations satisfy seven properties as follows.
\begin{itemize}
\item[$A_1$]\emph{Associativity of products.} For any $c\in H^*(X\ra Y),d\in H^*(Y\ra Z),e\in H^*(Z\ra W)$ we have
\[(c\cdot d)\cdot e=c\cdot (d\cdot e)\in H^*(X\ra W).\]

\item[$A_2$]\emph{Functoriality of push-forwards.} Let $f:X\ra Y$ and $g:Y\ra Z$ be proper.  For any $Z\ra W$ and $c\in H^*(X\ra W)$ we have
\[(g\circ f)_*(c)=g_*(f_*(c))\in H^*(Z\ra W).\]

\item[$A_3$]\emph{Functoriality of pull-backs.} For any $c\in H^*(X\ra Y)$, $g:Y_1\ra Y$, $h:Y_2\ra Y_1$, we have 
\[(g\circ h)^*(c)=h^*g^*(c)\in H^*(X_2\ra Y_2).\]
where $X_2=X\times_Y Y_2.$

\item[$A_4$]\emph{Product and push-forward commute.} Let $f:X\ra Y$ be proper.  For any $Y\ra Z$, $Z\ra W$ and $c\in H^*(X\ra Z),d\in H^*(Z\ra W)$, we have
\[f_*(c)\cdot d=f_*(c\cdot d)\in H^*(Y\ra W).\]

\item[$A_5$]\emph{Product and pull-back commute.} We have the following Cartesian diagram
\begin{equation}\label{ppbdia}\xymatrix{X'\ar[r]^{f'}\ar[d]&Y'\ar[d]^{g'}\ar[r]&Z'\ar[d]^{g}\\X\ar[r]_f&Y\ar[r]&Z}.    \end{equation}
Then for any $c\in H^*(X\ra Y),d\in H^*(Y\ra Z)$
\[g^*(c\cdot d)=g'^*(c)\cdot g^*(d)\in H^*(X'\ra Z').\]

\item[$A_6$]\emph{Push-forward and pull-back commute.}  In the diagram (\ref{ppbdia}), if $f:X\ra Y$ is proper, then for any $c\in H^*(X\ra Z)$, we have
\[g^*f_*(c)=f'_*(g^*(c))\in H_*(Y'\ra Z').\]

\item[$A_7$]\emph{Projection formula.}  Given a morphism $h:Y\ra Z$ and a Cartesian diagram
 \[\xymatrix{X'\ar[r]^{f'}\ar[d]_{g'}&Y'\ar[d]^g\\X\ar[r]_f&Y},\]
with $g$ proper, we have for $c\in H^*(X\xrightarrow{f} Y),d\in H^*(Y'\xrightarrow{h\circ g}Z)$
\[c\cdot g_*(d)=g'_*(g^*(c)\cdot d)\in H^*(X\xrightarrow{h\circ f}Z).\] 
\end{itemize}
\begin{rem}\label{defbiv}One way to define the bivariant group is to use the relative cohomology.  Let $f:X\ra Y$ be any map.  Assume there is a closed embedding $\imath:X\hookrightarrow M$ with $M$ an oriented manifold.  Then we have a closed embedding $X\stackrel{(\imath,f)}{\hookrightarrow}M\times Y$, and we have
\[H^*(X\xrightarrow{f}Y):=H^*(Y\times M,Y\times M\setminus X).\]
\end{rem}
\begin{rem}\label{relhom}We would like to mention that $H^i(X\ra pt)$ and $H^i(X\xrightarrow{id} X)$ can be naturally identified with $H_{-i}^{BM}(X)$ and $H^i(X)$ for any space $X$.  Moreover, given $f:X\ra Y$ with $Y$ a nonsingular variety of dimension $d$, by \cite[Proposition 17.4.2, Example 17.4.3]{Ful} the product with the fundamental class $[Y]\in H^{-2d}(Y\ra pt)$ gives an isomorphism
\[H^i(X\xrightarrow{f}Y)\xrightarrow[\cong]{\cdot [Y]}H^{i-2d}(X\ra pt)=H_{2d-i}^{BM}(X).\]
\end{rem}
\begin{rem}[Orientations]\label{orien}For some maps $f:X\ra Y$ there are naturally determined elements in $H^*(X\xrightarrow{f} Y)$ called \emph{canonical orientations} and denoted $[f]$.  Orientations don't always exist in general but we have the following three special cases.
\begin{itemize}
\item[(1)] If $f:X\ra Y$ is flat of relative dimension $n$, then $[f]\in H^{-2n}(X\xrightarrow{f} Y)$ and for any $g:Y'\ra Y$, $a\in H^{BM}_*(Y')$, we have $$[f](a)=f^{'*}(a),$$
where $f':X'=X\times_Y Y'\ra Y'$ and $f^{'*}$ is the flat pullback.
\item[(2)] If $f:X\ra Y$ is a regular embedding of codimension $d$, then $[f]\in H^{2d}(X\xrightarrow{f} Y)$ and for any $g:Y'\ra Y$, $a\in H^{BM}_*(Y')$, we have $$[f](a)=f^{!}(a),$$
where $f^{!}$ is the refined Gysin homomorphism.  If $f':X'=X\times_Y Y'\ra Y'$ is also a regular embedding, then $f^{!}=f^{'*}$.
\item[(3)] If $f:X\ra Y$ is a l.c.i morphism with $f=p\circ i$ such that $p$ is smooth and $i$ is a regular embedding, then $$[f]:=[i]\cdot [p].$$
\end{itemize}
\end{rem}

\subsection{Riemann-Roch for singular varieties.}\label{appRR} In this subsection we recall very briefly the Riemann-Roch theory for singular varieties.  One can see more details in (\cite[Ch.18]{Ful}).   

For $X$ any algebraic scheme (locally of finite type and separated) and let $K_0(X)$ be the Grothendieck group of coherent sheaves on $X$.  We have a group homomorphism 
\[\tau=\tau_X:K_0(X)\ra A_*(X),\]
where for any $\alpha\in K_0(X)$, $\tau(\alpha)\in A_*(X)$ generalizes $ch(\alpha)td(T_X)$ for $X$ smooth. We have the \emph{Todd class} of $X$ defined by $Td(X):=\tau(\mo_X)\in A_*(X)$. We also denote $\tau(\alpha)$ ($ch(\beta)$, resp.) its image in $H_*(X)$ ($H^*(X)$, resp.) via the cycle map.    

In general, $\tau(\alpha)$ only lies in $H_*(X)$, but for some special cases $\tau(\alpha)$ lies in the image of cap product of $H^*(X)$ with $Td(X)$.  For instance when $\beta\in K^0(X)$, we have $$\tau(\beta)=ch(\beta)\cap Td(X),$$
where $K^0(X)$ is the subgroup of $K_0(X)$ generated by vector bundles and $ch(\beta)\in H^*(X)$ is the exponential Chern character of $\beta$.  

Moreover, $Td(X)$ can sometimes be lifted to $H^*(X)$, in other words, $\exists~ td(X)\in H^*(X)$ such that $Td(X)=td(X)\cap[X]$.  For instance let $X$ be a locally complete intersection, then we have $td(X)=td(T_X)$ with $T_X$ the virtual tangent bundle of $X$, and in this case we can define 
\begin{equation}\label{defch}ch^{\tau}(G):=td(-T_{X})\cap\tau(\alpha)\in H_*(X)\end{equation}
for any $\alpha\in K_0(X)$.

We rephrase \cite[Theorem 18.3, Example 18.3.11 and Example 18.3.12]{Ful} as the following theorem.
\begin{thm}\label{Ful18.3}The homomorphism $\tau$ has the following five properties.
\begin{itemize}
\item[(1)](Covariance). If $f:X\ra Y$ is proper, $\alpha\in K_0(X)$, then $f_*\tau_X(\alpha)=\tau_Y(Rf_*(\alpha))$.
\item[(2)](Module). If $\alpha\in K_0(X),\beta\in K^0(X)$, then $$\tau(\alpha\otimes\beta)=ch(\beta)\cap\tau(\alpha).$$
\item[(3)] Let $i:X\ra M$ be a closed embedding and let $G$ be a coherent sheaf on $X$.  Assume there is a locally free resolution $E_{\bullet}$ of finite length of $i_* G$, then
\[\tau(G)=ch^M_X(E_{\bullet})\cap Td(M),\]
where $ch_X^M(E_{\bullet})$ is a bivariant class in $H^*(X\xrightarrow{i} M)$.  In particular if $M$ is smooth, then
\[\tau(G)=ch^M_X(E_{\bullet})\cap Td(M)=td(i^*T_M)\cap(ch^M_X(E_{\bullet})\cap[M]).\]

Moreover, for any coherent sheaf $\mg$ on $M$, $\mh^i(E_{\bullet}\otimes \mg)$ are sheaves on $X$ and we have
\[\sum_{i}(-1)^i\tau_X(\mh^i(E_{\bullet}\otimes \mg))=ch_X^M(E_{\bullet})\cap \tau_M(\mg).\] 

\item[(4)] Let $f:X\ra Y$ be a l.c.i. morphism such that both $X,Y$ admit closed embeddings to smooth schemes.  Then for all $\alpha\in K_0(Y)$
\[\tau_X(Lf^*\alpha)=td(T_f)\cap f^*\tau_Y(\alpha).\]
If moreover both $X,Y$ are locally complete intersection, we have
\[ch^{\tau_X}(Lf^*\alpha)=f^*ch^{\tau_Y}(\alpha).\]
\item[(5)](Top term). If $Z$ is a closed subvariety of $X$, with $\dim Z=n$. then 
\[\tau(\mo_Z)=[V]+\text{terms of dimension }<n.\]

Moreover let $G$ be a coherent sheaf on $X$ with $\dim Supp(G)\leq n$.  Define the $n$-cycle of $G$, 
\[Z_n(G)=\sum_{\dim V=n}l_V(G)[V]\]
where the sum is over all $n$-dimensional subvarieties $V$ in the support of $G$, and $l_V(G)$ is the length of the stalk of $G$ over the local ring of $X$ at $V$.  Then
\[\tau(G)=Z_n(G)+\text{terms of dimension }<n.\]
\end{itemize}
\end{thm}

\section{Proof of Proposition \ref{chen}}\label{appA}
\begin{proof}[Proof of Proposition \ref{chen}]We have the closed embedding $J\hookrightarrow \mj$ with $\mj$ smooth.  By definition of $E_n$ we can see that $E_n$ is the restriction of a vector bundle $\me_n$ on $\mj$ to $J$. Therefore $ch_i(E_n)$ are pullbacks of cohomology classes on $\mj$ and hence either $ch_i(E_n)\not\in W_{< 2i}H^{2i}(J) $ or $ch_i(E_n)=0$ by mixed Hodge theory (e.g. \cite[Theorem 5.33 (iii)]{PeSt}).  

Let $f:\widetilde{J}\ra J$ be the resolution of singularities.  Then by \cite[Corollary 5.42]{PeSt}, in order to show $ch_i(E_n)=0$ for all $i\geq 2$, it suffices to show that $f^*(ch_i(E_n))=0$ for all $i\geq 2$.  

We have the diagram
\[\xymatrix@C=1.5cm{&C\times \widetilde{J}\ar[ld]_{p_1}\ar[d]_{id_C\times f}\ar[rd]^{\widetilde{p}_2}&\\
C & C\times J &\widetilde{J}. }\]
Since $F$ is flat over both factors by \cite[Theorem A(2), Lemma 6.1 (3)]{Ari2}, $(id_C\times f)^*F$ is also the derived pullback.  By the universal property of the $E_n$, we have $f^*(ch_i(E_n))=ch_i(f^*(E_n))=ch_i(\widetilde{p}_{2,*}(p_1^*\mo_C(nx)\otimes (id_C\times f)^*F))$.

Since $C$ and $J$ are all locally complete intersections, $ch^{\tau}(G)$ is well-defined for any coherent sheaf $G$ on $C,J$ or $C\times J$.  By the covariance of $\tau$ (Theorem \ref{Ful18.3} (1)), we have 
\begin{equation}\label{rrs}ch(E_n)\cap [J]=ch^{\tau}(E_n)=p_{2,*}(p_1^*(ch(\mo_C(nx))\cup td(T_C))\cap ch^{\tau}(F)).\end{equation}

By a direct computation we have 
\begin{equation}\label{xcc}ch(\mo_C(nx))\cup td(T_C)=1+(n-g+1)x\in H^*(C).\end{equation}  Notice that here we take the class in $H^*(C)$ rather than $A^*(C)$.  On the other hand, since $F$ is normalized at $x$, we have
\begin{eqnarray}\label{xpf}p_{2,*}(p_1^*(x)\cap ch^{\tau}(F))&=&p_{2,*}(p_1^*(ch(\mo_x)\cup td(T_C))\cap ch^{\tau}(F))\nonumber\\&=&ch(F_x)\cap [J]=ch(\mo_J)\cap[J]=[J].\end{eqnarray}
Combine (\ref{rrs}), (\ref{xcc}) and (\ref{xpf}) and we have
\begin{equation}\label{cceng}ch(E_n)\cap [J]=ch^{\tau}(E_{2g-1})+(n-2g+1)[J]=p_{2,*}(ch^{\tau}(F))+(n-g+1)[J].\end{equation}

Analogously we have 
\begin{equation}\label{rrs1}ch(f^*E_n)\cap[\widetilde{J}]=\widetilde{p}_{2,*}(ch^{\tau}((id_C\times f)^*F))+(n-g+1)[\widetilde{J}].\end{equation}
Since $\widetilde{J}$ is smooth, it suffices to show that the degree $2(g-i)$ terms of $\widetilde{p}_{2,*}(ch^{\tau}((id_C\times f)^*F))\in H_*(\widetilde{J})$ are zero for all $i\geq 2$.  We will prove this statement in three steps.

\emph{Step 1.} Let $h:\widetilde{C}\ra C$ be the normalization of $C$.  We have the following exact sequence
\[0\ra \mo_C\ra h_*\mo_{\widetilde{C}}\ra T\ra0,\]
where $T$ is a zero-dimensional sheaf supported at the singular locus of $C$.

Since $F$ is flat over both factors, $(id_C\times f)^*F$ is flat over $C$ and on $C\times \widetilde{J}$ we have the following exact sequence 
\begin{equation}\label{xcc1}0\ra (id_C\times f)^*F\ra (h\times id_{\widetilde{J}})_*(\widetilde{F})\ra (id_C\times f)^*F\otimes p_1^*T\ra 0,\end{equation}
where $\widetilde{F}:=(h\times f)^*F=(h\times id_{\widetilde{J}})^*(id_C\times f)^*F$.

We have the diagram
\[\xymatrix@C=1.5cm{&\widetilde{C}\times \widetilde{J}\ar[ld]_{\widetilde{p}_1}\ar[d]_{h\times id_{\widetilde{J}}}\ar[rd]^{\widetilde{\widetilde{p}}_2}&\\
\widetilde{C}\ar[d]_h & C\times \widetilde{J}\ar[ld]_{p_1}\ar[d]_{id_C\times f}\ar[r]_{\widetilde{p}_2} &\widetilde{J}\\C& C\times J &. }\]

By (\ref{xcc1}) we have 
\begin{equation}\label{xcc2}\widetilde{p}_{2,*}(ch^{\tau}((id_C\times f)^*F))=\widetilde{p}_{2,*}(ch^{\tau}((h\times id_{\widetilde{J}})_*(\widetilde{F})))-\widetilde{p}_{2,*}(ch^{\tau}((id_C\times f)^*F\otimes p_1^*T)).
\end{equation}

\emph{Step 2.} We claim that 
\begin{equation}\label{tts}\widetilde{p}_{2,*}(ch^{\tau}((id_C\times f)^*F\otimes p_1^*T))=\ell[\widetilde{J}]\end{equation} with $\ell:=g-g_{\widetilde{C}}$ the length of $T$.  Since $T$ is zero-dimensional, we have 
$$[T]=[\mo_{x_1}]+\cdots+[\mo_{x_{\ell}}]\in K_0(C).$$
Therefore it suffices to show $\widetilde{p}_{2,*}(ch^{\tau}((id_C\times f)^*F\otimes p_1^*\mo_{y}))=[\widetilde{J}]$ for any $y\in C$.  Since $(id_C\times f)^*F\otimes p_1^*\mo_{y}$ is supported at $\widetilde{J}\times {y}$, by the top term property of $\tau$ (Theorem \ref{Ful18.3} (5)) we have 
\[p_1^*x\cap ch^{\tau}((id_C\times f)^*F\otimes p_1^*\mo_{y})=0,\]
and hence \[p_1^*(td(T_C))\cap ch^{\tau}((id_C\times f)^*F\otimes p_1^*\mo_{y})=ch^{\tau}((id_C\times f)^*F\otimes p_1^*\mo_{y}).\]
Therefore by covariance of $\tau$ we have
\begin{eqnarray}\widetilde{p}_{2,*}(ch^{\tau}((id_C\times f)^*F\otimes p_1^*\mo_{y}))&=&\widetilde{p}_{2,*}(p_1^*(td(T_C))\cap ch^{\tau}((id_C\times f)^*F\otimes p_1^*\mo_{y}))\nonumber\\
&=&ch^{\tau}(R\widetilde{p}_{2,*}((id_C\times f)^*F\otimes p_1^*\mo_{y}))\nonumber\\
&=&ch^{\tau}(f^*F_y)=[\widetilde{J}],\nonumber
\end{eqnarray}
where the last equality is because $f^*F_x\cong f^*\mo_J\cong\mo_{\widetilde{J}}$ and $f^*F_y$ lies in the same flat family with $f^*F_x$ parametrized by the connected scheme $C$.

By (\ref{xcc2}), it suffices to show that the degree $2(g-i)$ terms of $\widetilde{p}_{2,*}(ch^{\tau}((h\times id_{\widetilde{J}})_*\widetilde{F}))\in H_*(\widetilde{J})$ are zero for all $i\geq 2$.

\emph{Step 3.} Since $h$ is affine, $(h\times id_{\widetilde{J}})_*$ is the derived pushforward and by the covariance of $\tau$ we have 
\begin{eqnarray}\label{xcc3}ch^{\tau}((h\times id_{\widetilde{J}})_*(\widetilde{F}))&=&\tau((h\times id_{\widetilde{J}})_*\widetilde{F})p_1^*td(-T_{C})\widetilde{p}_2^*td(-T_{\widetilde{J}})\nonumber\\
&=&(h\times id_{\widetilde{J}})_*(\tau(\widetilde{F}))p_1^*td(-T_{C})\widetilde{p}_2^*td(-T_{\widetilde{J}})\nonumber\\
&=& (h\times id_{\widetilde{J}})_*(ch^{\tau}(\widetilde{F})\widetilde{p}_1^*(td(T_{\widetilde{C}}))\widetilde{\widetilde{p}}_2^*(td(T_{\widetilde{J}})))p_1^*td(-T_{C})\widetilde{p}_2^*td(-T_{\widetilde{J}})\nonumber\\
&=&(h\times id_{\widetilde{J}})_*(ch^{\tau}(\widetilde{F})\widetilde{p}_1^*(td(T_{\widetilde{C}})))p_1^*td(-T_{C})
\end{eqnarray}
where the last equality is obtained by the projection formula and $\widetilde{\widetilde{p}}_2=\widetilde{p}_2\circ (h\times id_{\widetilde{J}})$.

Since $x$ is a smooth point in $C$, we may also denote $x$ its preimage in $\widetilde{C}$.  Then $\widetilde{F}_x=f^*(F_x)\cong f^*(\mo_J)\cong\mo_{\widetilde{J}}$.   Denote by $U\subset C\times J$ the biggest open subset where $F$ is locally free.  Then the complement of $U$ in $C\times J$ is of codimension $\geq 2$.  Let $V:=(h\times f)^{-1}(U)\subset \widetilde{C}\times \widetilde{J}$.  Then the complement of $V$ in $\widetilde{C}\times \widetilde{J}$ is also of codimension $\geq 2$, this is because $U$ contains $(C\times J^{sm})\cup (C^{sm}\times J)$.  

Let $\imath:U\hookrightarrow C\times J,~\jmath:V\hookrightarrow \widetilde{C}\times \widetilde{J}$ be the two embeddings.  Since $F$ is Cohen-Macaulay,  $F=\imath_*(F|_U)$ by 
(e.g. \cite[Lemma 2.2]{Ari2}).  Since both $C\times J$ and $\widetilde{C}\times \widetilde{J}$ are Cohen-Macaulay, by an analogous argument of Lemma \ref{torf} we have that 
$$\widetilde{F}=(h\times f)^*F=(h\times f)^*(\imath_*(F|_U))=\jmath_*((h\times f)^*(F|_U))=\jmath_*(\widetilde{F}|_V).$$
On the other hand, $F|_U$ is a line bundle and hence so is $\widetilde{F}|_V$.  By the smoothness of $\widetilde{C}\times \widetilde{J}$, we have $\widetilde{F}\cong \det \widetilde{F}$ is a line bundle on $\widetilde{C}\times \widetilde{J}$.  Therefore
\begin{equation}\label{xpf1}ch^{\tau}(\widetilde{F})=ch(\widetilde{F})\cap [\widetilde{C}\times\widetilde{J}]=Exp(c_1(\widetilde{F}))\cap [\widetilde{C}\times\widetilde{J}].
\end{equation}
By K\"unneth Formula, we write $c_1(\widetilde{F})=a_0\otimes b_2+a_1\otimes b_1+a_2\otimes b_0$ where $a_i\in H^i(\widetilde{C}),~b_i\in H^i(\widetilde{J}).$  Because $\widetilde{F}_x\cong \mo_{\widetilde{J}}$, by Riemann-Roch and a direct computation we have
$(a_0\cdot x)b_2=c_1(\widetilde{F}_x)=0$, hence $a_0\otimes b_2=0$. Therefore we have
\[ch^{\tau}(\widetilde{F})=(1+a_1\otimes b_1+a_2\otimes b_2+\frac12a_1^2\otimes b_1^2)\cap[\widetilde{C}\times\widetilde{J}].\]
Since $x\cup a_i=0$ for $i=1,2$, we have
\begin{equation}ch^{\tau}(\widetilde{F})\widetilde{p}_1^*(td(T_{\widetilde{C}}))=(1+a_1\otimes b_1+a_2\otimes b_2+\frac12a_1^2\otimes b_1^2)\cap[\widetilde{C}\times\widetilde{J}]+(1-g_{\widetilde{C}})[x\times\widetilde{J}].
\end{equation}
By (\ref{xcc3}) we have
\begin{eqnarray}\label{xcc}ch^{\tau}((h\times id_{\widetilde{J}})_*(\widetilde{F}))&=&(h\times id_{\widetilde{J}})_*(ch^{\tau}(\widetilde{F})\widetilde{p}_1^*(td(T_{\widetilde{C}})))p_1^*td(-T_{C})\nonumber\\
&=&(h\times id_{\widetilde{J}})_*((1+a_1\otimes b_1+a_2\otimes b_2+\frac12a_1^2\otimes b_1^2)\cap[\widetilde{C}\times\widetilde{J}])\nonumber\\
&+&(g-g_{\widetilde{C}})[x\times \widetilde{J}].
\end{eqnarray}
Therefore we have
\begin{equation}\label{xpf2}\widetilde{p}_{2,*}(ch^{\tau}((h\times id_{\widetilde{J}})_*\widetilde{F}))=(g-g_{\widetilde{C}})[\widetilde{J}]+(a_2\cap[\widetilde{C}])b_2\cap[\widetilde{J}]+(\frac12a_1^2\cap[\widetilde{C}])b_1^2\cap[\widetilde{J}].\end{equation}
The RHS of (\ref{xpf2}) only contains non-trivial terms in degree $2g$ and $2g-2$.

The last thing left to prove is that $c_1(E_n)$ is independent of $n$.  This can be seen from the fact that $f^*ch_1(E_n)$ is independent of $n$ by (\ref{rrs1}), (\ref{xcc2}), (\ref{tts}) and (\ref{xpf2}).  The proposition is proved.
\end{proof}


Yao Yuan\\
Beijing National Center for Applied Mathematics,\\
Academy for Multidisciplinary Studies, \\
Capital Normal University, 100048, Beijing, China\\
E-mail: 6891@cnu.edu.cn.
\end{document}